\RequirePackage{xcolor}
\documentclass[journal,twoside,web]{IEEEcolor}
\usepackage{generic}
\usepackage{cite}
\usepackage{amsmath,amssymb,amsfonts}
\usepackage{graphicx}
\usepackage{textcomp}
\usepackage{graphbox}
\usepackage{hyperref}
\usepackage{dsfont,subfigure,hyperref,fancyhdr}
\usepackage{mathtools,mathrsfs,wrapfig}
\usepackage{ragged2e}
\usepackage{comment}
\usepackage{xspace}
\usepackage{scalerel}

\usepackage{algorithm}
\usepackage[noend]{algorithmic}
\algsetup{indent=2em}

\let\labelindent\relax
\usepackage{enumitem}

\allowdisplaybreaks

\usepackage{version}

\excludeversion{tac}
\includeversion{arxiv}
\excludeversion{extra}
\excludeversion{outfornow}

\usepackage{mdframed}
\newmdtheoremenv{prob}{Problem}

\newtheorem{theorem}{Theorem}

\newtheorem{lemma}[theorem]{Lemma}
\newtheorem{remark}[theorem]{Remark}
\newtheorem{prop}[theorem]{Proposition}
\newtheorem{defn}[theorem]{Definition}

\newtheorem{assumption}{Assumption}
\newtheorem{example}{Example}

\newcommand{\R}{\mathbb{R}}

\newcommand{\cs}{c^\star}
\newcommand{\as}{\alpha^\star}

\definecolor{gnred}{RGB}{255,91,89}

\definecolor{gnred1}{RGB}{71,0,0} 
\definecolor{gnred2}{RGB}{117,0,0} 
\definecolor{gnred3}{RGB}{164,0,0} 
\definecolor{gnred4}{RGB}{211,0,0} 
\definecolor{gnred5}{RGB}{255,0,0} 
\definecolor{gnred6}{RGB}{255,42,34} 
\definecolor{gnred7}{RGB}{255,91,89} 

\definecolor{gnblue1}{RGB}{0,36,71}   
\definecolor{gnblue2}{RGB}{0,60,118}  
\definecolor{gnblue3}{RGB}{0,85,164}
\definecolor{gnblue4}{RGB}{0,108,212}
\definecolor{gnblue4}{RGB}{0,108,212}
\definecolor{gnblue5}{RGB}{0,133,255}  
\definecolor{gnblue6}{RGB}{35,156,255} 
\definecolor{gnblue7}{RGB}{88,177,255} 

\definecolor{gnbrown1}{RGB}{71,27,0}  
\definecolor{gnbrown2}{RGB}{117,45,0} 
\definecolor{gnbrown3}{RGB}{164,62,0} 
\definecolor{gnbrown4}{RGB}{211,80,0} 
\definecolor{gnbrown5}{RGB}{255,97,0} 
\definecolor{gnbrown6}{RGB}{255,127,26} 
\definecolor{gnbrown7}{RGB}{255,155,86} 

\usepackage{hyperref}%
\hypersetup{colorlinks=true, linkcolor=blue, breaklinks=true, urlcolor=blue}

\newcommand{\until}[1]{\{1,\dots, #1\}}
\newcommand{\subscr}[2]{#1_{\textup{#2}}}

\newcommand{\setdef}[2]{\{#1 \; | \; #2\}}

\newcommand{\map}[3]{#1: #2 \rightarrow #3}

\newcommand{\e}{\mathrm{e}}

\newcommand{\real}{\ensuremath{\mathbb{R}}}
\newcommand{\realpositive}{\ensuremath{\mathbb{R}}_{>0}}
\newcommand{\realnonnegative}{\ensuremath{\mathbb{R}}_{\ge 0}}

\newcommand{\realextended}{\overline{\real}}

\newcommand\oprocendsymbol{\hbox{$\triangle$}}
\newcommand\oprocend{\relax\ifmmode\else\unskip\hfill\fi\oprocendsymbol}

\renewcommand{\theenumi}{(\roman{enumi})}

\DeclareSymbolFont{bbold}{U}{bbold}{m}{n}
\DeclareSymbolFontAlphabet{\mathbbold}{bbold}
\newcommand{\vect}[1]{\mathbbold{#1}}

\newcommand{\vectorzeros}[1][]{\vect{0}_{#1}}

\newcommand{\ds}{\displaystyle}

\newcounter{saveenum}

\usepackage[prependcaption,colorinlistoftodos]{todonotes}

\newcommand*{\mydoi}[1]{\href{http://dx.doi.org/#1}{\includegraphics[width=.75em]{doi.png}}}
\newcommand{\jac}[1]{D\mkern-0.75mu{#1}}
\newcommand{\jacx}[1]{D_x\mkern-0.5mu{#1}}
\newcommand{\jacu}[1]{D_{\theta}\mkern-0.5mu{#1}}

\newcommand{\GB}{Gr\"onwall\xspace}
\newcommand{\seminorm}[1]{{\left\vert\kern-0.25ex\left\vert\kern-0.25ex\left\vert #1
		\right\vert\kern-0.25ex\right\vert\kern-0.25ex\right\vert}}

\newcommand{\semimeasure}[1]{\mu_{\seminorm{\cdot}}\kern-0.5ex\left(#1\right)}

\newcommand{\osL}{\operatorname{osL}}
\newcommand{\Lip}{\operatorname{Lip}}
\newcommand{\realpart}{\operatorname{{Re}}}

\renewcommand{\realpart}{\mathrm{Re}}

\newcommand{\norm}[2]{\|#1\|_{#2}}

\DeclareMathOperator*{\argmin}{arg\,min}
\DeclareMathOperator{\proj}{P}

\newcommand{\Id}{\mathsf{Id}}

\newcommand{\OF}{\mathsf{F}}

\newcommand{\OT}{\mathsf{T}}

\newcommand{\mcX}{\mathcal{X}}
\newcommand{\mcU}{\mathcal{U}}
\newcommand{\mcY}{\mathcal{Y}}

\newcommand{\mcB}{\mathcal{B}}
\newcommand{\mcC}{\mathcal{C}}

\newcommand{\xstar}{x^{\star}}
\newcommand{\xstardot}{\dot{x}^{\star}}
\newcommand{\fpal}{\subscr{\OF}{PAL}}

\newcommand{\prox}[1]{\mathrm{prox}_{#1}}
\newcommand{\VI}[2]{\mathrm{VI}(#1,#2)}
\newcommand{\Hess}{\nabla^2}
\newcommand{\relu}{\mathrm{ReLU}}

\newcommand{\bmin}{\rho}
\newcommand{\bmax}{\ell}
\newcommand{\amin}{\subscr{a}{min}}
\newcommand{\amax}{\subscr{a}{max}}

\newcommand{\lognorm}[2]{\mu_{#2}(#1)}

\newcommand{\mf}{\rho}
\newcommand{\lf}{\ell}
\newcommand{\trasp}[1]{#1^{\top}}

\newcommand{\0}{\mbox{\fontencoding{U}\fontfamily{bbold}\selectfont0}}

\newcommand{\z}[1]{\0_{#1}}
\newcommand{\fzero}{\0}

\newcommand{\change}[1]{\textcolor{black}{#1}}
\newcommand{\oldchange}[1]{\textcolor{black}{#1}}
\newcommand{\changeafter}[1]{\textcolor{black}{#1}}

\title{Time-Varying Convex Optimization: A Contraction and Equilibrium Tracking Approach}

\author{Alexander Davydov$^{a}$, \and Veronica Centorrino$^{b}$, \and Anand Gokhale$^{a}$, \and Giovanni Russo$^{c}$, \and Francesco Bullo$^{a}$
\thanks{This work was in part supported by NSF Graduate Research Fellowship under Grant No. 2139319 and AFOSR project FA9550-21-1-0203. \change{GR was supported by the European Union-Next Generation EU Mission 4 Component 1 CUP E53D23014640001.} \change{The last author thanks Ryotaro Shima for insightful discussions.}}%
\thanks{$^{a}$ Center for Control, Dynamical 
Systems, and Computation, UC Santa Barbara, Santa Barbara, CA 93106 USA. {\tt\small{\{davydov, anand\_gokhale, bullo\}@ucsb.edu}}
}
\thanks{$^{b}$ Scuola Superiore Meridionale, Naples, Italy. {\tt\small veronica.centorrino@unina.it}.}
\thanks{$^{c}$ Department of Information and Electric Engineering and Applied Mathematics, University of Salerno, Italy. {\tt\small giovarusso@unisa.it}.}
}
\date{\today}

\begin{document}

\maketitle
\begin{abstract}
 In this article, we provide a novel and broadly-applicable
 contraction-theoretic approach to continuous-time time-varying convex
 optimization.  For any parameter-dependent contracting dynamics, we show
 that the tracking error is asymptotically proportional to the rate of
 change of the parameter \change{and that} the proportionality constant is
 upper bounded by Lipschitz constant in which the parameter appears divided
 by the contraction rate of the dynamics squared.  \change{We additionally
   establish that augmenting any parameter-dependent contracting dynamics
   with a feedforward prediction term ensures that the tracking error
   vanishes exponentially quickly.}  To apply these results to
 time-varying convex optimization, we establish the strong infinitesimal
 contractivity of dynamics solving three canonical problems: monotone
 inclusions, linear equality-constrained problems, and composite
 minimization problems. For each \change{case}, we \change{derive} the
 sharpest-known contraction rates and provide explicit bounds on the
 tracking error between solution trajectories and minimizing
 trajectories. We validate our theoretical results on \change{two}
 numerical examples \change{and on an} application to control barrier
   function-based controller design \change{that involves real hardware}.
\end{abstract}


\section{Introduction}

\textit{Problem description and motivation: }Mathematical optimization is a fundamental tool in science and engineering research and has pervaded countless application areas. The classical perspective on mathematical optimization is numerical and is motivated through implementation of iterative algorithms on digital devices. An alternative perspective is to view optimization algorithms as dynamical systems and to understand the performance of these algorithms via their dynamical systems properties, e.g., stability and robustness. 

Studying optimization algorithms as continuous-time dynamical systems has been an active area of research since the seminal work of Arrow, Hurwicz, and Uzawa~\cite{KJA-LH-HU:58}. 
Notable examples include Hopfield and Tank in dynamical neuroscience~\cite{JJH-DWT:85}, Kennedy and Chua in analog circuit design~\cite{MPK-LOC:88}, and Brockett in systems and control~\cite{RWB:91}. Recently, the interest in continuous-time dynamics for optimization and computation has been renewed due to the advent of (i) online and dynamic feedback optimization~\cite{GB-JC-JIP-EDA:22}, (ii) reservoir computing~\cite{GT-TY-JBH-RN-NK-ST-HN-DN-AH:19}, and (iii) neuromorphic computing~\cite{CDS-SRK-MP-JPM-PD:22}. 

Motivated by these recent developments, we are interested in time-varying convex optimization problems and continuous-time dynamical systems that track their optimal solutions. In many applications of interest, the optimization algorithm must be run in real-time on problems that are time-varying. Such examples include tracking a moving target, estimating the path of a stochastic process, and online learning. In such application areas, we would like our dynamical system to converge to the unique optimal solution when the problem is time-invariant and converge to an explicitly-computable neighborhood of the optimal solution trajectory when the problem is time-varying.

Beyond tracking optimal trajectories, a key desirable feature of optimization algorithms is robustness in the face of uncertainty. In many real-world scenarios, we are not provided the exact value of our cost function but instead noisy estimates and possibly even a time-delayed version of it. Thus, for practical usage of the optimization algorithm, it is essential to ensure that the algorithm has these robustness features \emph{built-in}. 

Remarkably, all of these desirable properties, namely  tracking for time-varying systems, convergence for time-invariant systems, and robustness to noise and time-delays
can be established by ensuring that the dynamical system is \emph{strongly infinitesimally contracting}~\cite{WL-JJES:98}. To be specific, (i) the effect of the initial condition is exponentially forgotten and the distance between any two trajectories decays exponentially quickly~\cite{WL-JJES:98}, (ii) if a contracting system is time-invariant, it has a unique globally exponentially stable equilibrium point~\cite{WL-JJES:98}, (iii) contracting systems are incrementally input-to-state stable and thus robust to disturbances~\cite{SX-GR-RHM:21} and delays~\cite{PHAN-HT:18}.

\textit{Literature review:} 
A recent survey on studying optimization algorithms from a feedback control perspective is available in~\cite{AH-ZH-SB-GH-FD:24}.
Asymptotic and exponential stability of dynamical systems solving convex optimization problems is a classical problem and has been studied in papers including~\cite{JW-NE:11,AC-BG-JC:17,GQ-NL:19,JC-SKN:19,NKD-SZK-MRJ:19} among many others. Compared to papers studying asymptotic and exponential stability, there are far fewer works studying the contractivity of dynamical systems solving optimization problems. A few exceptions include~\cite{HDN-TLV-KT-JJES:18,PCV-SJ-FB:19r} which analyze primal-dual dynamics and \cite{PMW-JJES:20,JWSP-FB:12za} which study gradient flows on Riemannian manifolds.

In the context of time-varying convex optimization, algorithms to track the optimal solution are designed based on Newton's method in (i) discrete-time in~\cite{AS-AM-AK-GL-AR:16,AS-EDA:17} and in (ii) continuous-time in~\cite{MF-SP-VMP-AR:18}.
See both~\cite{EDA-AS-SB-LM:20} and~\cite{AS-EDA-SP-GL-GBG:20} and the references therein for reviews of these results and theoretical extensions. These results have been leveraged to study the feedback interconnection of a LTI system and a dynamical system solving an optimization problem in~\cite{MC-ED-AB:20}. From a contraction theory perspective, both~\cite{HDN-TLV-KT-JJES:18} and~\cite{PCV-SJ-FB:19r} provide tracking error bounds for continuous-time time-varying primal-dual dynamics.

\textit{Contributions:}  
This paper makes four main contributions. First, we prove a general theorem regarding parameter-dependent strongly infinitesimally contracting dynamics. Specifically, in Theorem~\ref{thm:bound_time_varying_eq}, we prove that \change{both} the tracking error, \change{defined as the error between any solution trajectory and the equilibrium trajectory (defined instantaneously in time), and the norm of the vector field are uniformly upper bounded}. 
Moreover, we prove that the tracking error is asymptotically proportional to the rate of change of the parameter with proportionality constant upper bounded by $\ell_\theta/c^2$ where $\ell_\theta$ is the Lipschitz constant in which the parameter appears and $c$ is the contraction rate of the dynamics.
A related result was proved in~\cite[Lemma~2]{HDN-TLV-KT-JJES:18}, but the tracking error bound depends on the knowledge of the rate of change of the equilibrium trajectory, which is unknown, in general. \change{In contrast,} Theorem~\ref{thm:bound_time_varying_eq} provides \change{an additional bound on the norm of the vector field and a tracking error bound} which depends purely on the rate of change of the parameter, \change{which may be more directly applicable}. 

\oldchange{Second, in Theorem~\ref{thm:exact=tracking}, we propose an alternative dynamical system which augments the contracting dynamics in Theorem~\ref{thm:bound_time_varying_eq} with a feedforward term. This augmentation ensures that the tracking error is exponentially decaying to zero and does not require any Lipschitz condition on how the parameter appears in the dynamics.} \change{A related, continuous-time, treatment is proposed in~\cite{MF-SP-VMP-AR:18}, (see also the early reference~\cite{YZ-MNSS:98}) where the authors study a continuous-time Newton method and show how to add a feedforward term to ensure zero tracking error in the Euclidean norm. In discrete-time, the authors of~\cite{AS-AM-AK-GL-AR:16,AS-EDA:17} use predictor-corrector methods based on Newton's method with or without projections. Compared to these references, Theorem~\ref{thm:exact=tracking} is applicable to any contracting dynamics with respect to any norm and need not be limited to the solution of a time-varying optimization problem.}

\oldchange{Third}, we consider natural transcriptions into contracting dynamics for three canonical strongly convex optimization problems (namely, (i) monotone inclusions, (ii) linear equality-constrained problems, and (iii) composite minimization), and we make specific contributions for each transcription.  
\begin{enumerate}
\item For monotone inclusion problems, we \change{consider the  \emph{forward-backward splitting dynamics} which were first studied in~\cite{BA-HA:15}. These dynamics are} a generalization of the projected dynamics studied in~\cite{XBG:03} and the proximal gradient dynamics studied in~\cite{SHM-MRJ:21}. In Theorem~\ref{thm:contractivity-forwardbackward}, we show that the forward-backward splitting dynamics are contracting, a stronger property than exponential stability as was shown in~\cite[Theorem~4]{XBG:03} and~\cite[Theorem~2]{SHM-MRJ:21} and show improved rates of exponential convergence in some special cases.
  
    \item For linear-equality constrained problems, we study the primal-dual dynamics and prove their strong infinitesimal contractivity in Theorem~\ref{thm:primal-dual-contractivity}. Compared to~\cite{HDN-TLV-KT-JJES:18}, we provide an explicit estimate on the rate of contraction and we show improved rates compared to both~\cite{GQ-NL:19,PCV-SJ-FB:19r}.
    \item For composite minimization, we adopt the proximal augmented Lagrangian approach from~\cite{NKD-SZK-MRJ:19}, first introduced in~\cite{MF:75}, and show that the primal-dual dynamics on the proximal augmented Lagrangian are contracting in Theorem~\ref{thm:contractivity-pal}. This result improves on the exponential convergence result from~\cite[Theorem~3]{NKD-SZK-MRJ:19} by allowing for a larger range of parameters. A related result is~\cite[Theorem~2]{GQ-NL:19}, which focuses specifically on inequality-constrained minimization problems. Theorem~\ref{thm:contractivity-pal} is a generalization of~\cite[Theorem~2]{GQ-NL:19} to more general composite minimization problems and provides a nonlinear program to estimate the contraction rate. Moreover, to approximate the optimal value of the nonlinear program, we provide a strategy based on a bisection algorithm.
\end{enumerate}

\oldchange{Finally}, we apply our general result on tracking error bounds for contracting dynamics from \oldchange{Theorems~\ref{thm:bound_time_varying_eq} and~\ref{thm:exact=tracking}} to each of the aforementioned optimization problems to provide tracking error estimates in time-varying convex optimization problems. \oldchange{To validate our theory, we present numerical \change{and hardware} experiments. In Sections~\ref{sec:equality-ex} and~\ref{sec:inequality-ex} we showcase tracking error bounds for time-varying equality and inequality-constrained minimization problems, respectively. In Section~\ref{sec:CBF}, inspired by~\cite{ZM-FB-AGA:23r}, we present a modern application to online control barrier functions,~\cite{ADA-XX-JWG-PT:17}, where we show how we can leverage our tracking error results for contracting dynamics to ensure safety in a multi-robot collision avoidance scenario without needing to solve a quadratic program at every instance in time.}

\section{Notation and Preliminaries}
\paragraph{Notation}
Define $\realextended := [{-}\infty,+\infty]$. We let $\0_n \in \real^n$ be the all-zeros vector, $I_n$ be the $n \times n$ identity matrix, and $\map{\Id}{\R^n}{\R^n}$ be the identity map.
For symmetric $A, B \in \real^{n\times n}$, $A \preceq B$ means $B-A$ is positive semidefinite.

We let $\| \cdot \|$ be both a norm on $\real^n$ and its corresponding induced matrix norm on $\real^{n \times n}$. 
Given $A \in \R^{n \times n}$ the \emph{logarithmic norm} (log-norm) induced by $\| \cdot \|$ is defined by
$
\mu(A) := \lim_{h\to 0^{+}} \frac{\|I_n + h A\| - 1}{h}.
$
Given a symmetric positive-definite matrix $P \in \real^{n\times n}$, we let $\|\cdot\|_{P}$ be \oldchange{the} $P$-weighted $\ell_2$ norm $\|x\|_{P} := \sqrt{x^\top P x}$, $x \in \real^n$ and write $\|\cdot\|_2$ if $P = I_n$. The corresponding log-norm is $\mu_{P}(A) =\min\setdef{b \in \real}{PA + A^\top P \preceq 2bP}$~\cite[Lemma~2.7]{FB:24-CTDS}.

Given two normed spaces $(\mcX, \|\cdot\|_{\mcX})$, $(\mcY,\|\cdot\|_{\mcY})$ a map $\map{F}{\mcX}{\mcY}$ is Lipschitz from $(\mcX, \|\cdot\|_{\mcX})$ to $(\mcY,\|\cdot\|_{\mcY})$ with constant $\ell \geq 0$ if for all $x_1,x_2 \in \mcX$, it holds that
$
    \|F(x_1)-F(x_2)\|_{\mcY} \leq \ell\|x_1-x_2\|_{\mcX}.
$
If $\mcY = \mcX$ and $\|\cdot\|_{\mcX} = \|\cdot\|_{\mcY}$, we instead say $F$ is Lipschitz on $(\mcX,\|\cdot\|_{\mcX})$ with constant $\ell \geq 0$. A map $\map{F}{\real^n}{\real^n}$ is one-sided Lipschitz on $(\real^n, \|\cdot\|_P)$ with constant $b \in \real$ if for all $x_1,x_2 \in \real^n$ it holds that
$$
    (F(x_1) - F(x_2))^\top P(x_1 - x_2) \leq b\|x_1-x_2\|_P^2.
$$
We refer to~\cite[Section~3.2]{FB:24-CTDS} for generalizations of one-sided Lipschitz maps to more general norms.
We denote by $\Lip(F)$ the minimum \emph{Lipschitz constant} of $F$ and by 
$\osL(F)$ the \emph{one-sided Lipschitz constant} of $F$.
If $F$ is a multi-variable function we write $\Lip_{(\cdot)}(F)$ and $\osL_{(\cdot)}(F)$ to specify the variable with respect to which we are computing the Lipschitz and one-sided Lipschitz constant, respectively.
\oldchange{The \emph{upper-right Dini derivative} of a function $\map{\varphi}{\R}{\R}$ at $t$ is $D^+\varphi(t) := \limsup_{h\to 0^+} \bigl(\varphi(t+h) - \varphi(t)\bigr)/h$. The function $\map{\relu}{\R}{\R_{\geq 0}}$, is defined by $\relu(x) = \max\{0, x\}$.} 

\paragraph{Contraction theory} Given a continuous vector field $\map{F}{\real_{\geq 0} \times \real^n}{\real^n}$ with $(t,x) \mapsto F(t,x)$, a norm $\|\cdot\|$ on $\real^n$, and a constant $c >0$ ($c = 0)$ referred as \emph{contraction rate}, we say that $F$ is \emph{strongly (weakly) infinitesimally contracting with respect to $\|\cdot\|$ with rate $c$} if for all $t \geq 0$, the map $x \mapsto F(t,x)$ is one-sided Lipschitz with constant $-c$, i.e., $\osL_x(F) \leq -c$.  
For $F$ that is locally Lipschitz in $x$, $\osL_x(F) \leq -c$ if and only if $\mu(\jac{F}(t,x)) \leq -c$ for all $t \geq 0$ and almost every \oldchange{(a.e.)} $x \in \real^n$, where $\jac{F}(t,x):= \partial F(t,x)/\partial x$ is the Jacobian of $F$ with respect to $x$~\cite[Theorem~16]{AD-AVP-FB:22q}. 
Note that this Jacobian exists for \oldchange{a.e.} $x$ in view of Rademacher's theorem.
If $x(\cdot)$ and $y(\cdot)$ are two trajectories satisfying $\dot{x}(t) = F(t,x(t)), \dot{y}(t) = F(t,y(t))$, then $\|x(t) - y(t)\| \leq \e^{-c(t-t_0)}\|x(t_0) - y(t_0)\|$ for all $t \geq t_0 \geq 0$. 

One of the main benefits of contraction theory is that, with just a single condition, it ensures global
exponential convergence, along with other useful robustness properties. We refer to~\cite{FB:24-CTDS} for a recent review of \change{these} tools.

\paragraph{Convex analysis and monotone operators}
Let $\mcC \subseteq \real^n$ be convex, closed, and nonempty. \change{The} map $\map{\iota_{\mcC}}{\real^n}{\realextended}$ is the \emph{indicator function on $\mcC$} and is defined by $\iota_{\mcC}(z) = 0$ if $z \in \mcC$ and $\iota_{\mcC}(z) = +\infty$ otherwise. The map $\map{\proj_{\mcC}}{\R^n}{\mcC}$ is the \emph{projection \oldchange{on $\mcC$}} and is given by
$\proj_{\mcC}(x) = \argmin_{u \in \mcC} \|x-u\|_2$.

The epigraph of a map $\map{g}{\real^n}{\realextended}$ is the set $\setdef{(x,\xi) \in \real^{n+1}}{g(x) \leq \xi}$. 
The map $g$ is (i) \emph{convex} if its epigraph is a convex set, (ii) \emph{proper} if its value is never $-\infty$ and is finite somewhere, and (iii) \emph{closed} if it is proper and its epigraph is a closed set.
The map $\map{g}{\real^n}{\realextended}$ is (i) \emph{strongly convex with parameter $\mf > 0$} if the map $x \mapsto g(x) - \frac{\mf}{2}\|x\|_2^2$ is convex and (ii) strongly smooth with parameter $\lf \geq 0$ if it is differentiable and $\nabla g$ is Lipschitz on $(\real^n, \|\cdot\|_2)$ with constant $\lf$. 

Let $g$ be convex, closed, and proper (CCP). The \emph{subdifferential of $g$ at $x \in \real^n$} is the set $\partial g(x) := \setdef{z \in \real^n}{g(x) - g(y) \geq z^\top(x-y) \text{ for all } y \in \real^n}$. The \emph{proximal operator of $g$ with parameter $\gamma > 0$}, $\map{\prox{\gamma g}}{\real^n}{\real^n}$, is defined by
$$
\prox{\gamma g}(x) = \argmin_{z \in \real^n} g(z) + \frac{1}{2\gamma}\|x - z\|_2^2.
$$
The associated \emph{Moreau envelope \oldchange{of $g$ with parameter $\gamma > 0$}}, $\map{M_{\gamma g}}{\real^n}{\real}$, and its gradient are given by:
\begin{align}
M_{\gamma g}(x) &= g(\prox{\gamma g}(x)) + \frac{1}{2\gamma} \|x - \prox{\gamma g}(x)\|_2^2, \nonumber\\
\nabla M_{\gamma g}(x) &= \frac{1}{\gamma}(x - \prox{\gamma g}(x)). \label{eq:Moreau-gradient}
\end{align}
The gradient of the Moreau envelope always exists and is Lipschitz on $(\real^n, \|\cdot\|_2)$ with constant $1/\gamma$.

A map $\map{F}{\real^n}{\real^n}$ is (i) \emph{monotone} if $(F(x) - F(y))^\top (x-y) \geq 0$ for all $x, y \in \real^n$ and
(ii) \emph{strongly monotone with parameter $m > 0$} if the map $F - m \Id$ is monotone.
We refer to~\cite{HHB-PLC:17} for a comprehensive treatment of these tools.

\newcommand{\newparam}{\theta}
\newcommand{\newparamset}{\Theta}
\section{Equilibrium Tracking for Parameter-Varying Contracting Dynamical Systems}

\oldchange{We begin by considering a dynamical system which is a function of a time-varying parameter, $\theta$. Namely, 
for a vector field $\map{F}{\real^n \times \real^d}{\real^n}$, consider the system
  \begin{equation}
    \label{eq:system_1}
    \dot x(t) = F\bigl(x(t),\newparam(t)\bigr), \quad x(0) = x_0\in \real^n,
  \end{equation}
  where for all $t \geq 0$, $x(t)$ and $\newparam(t)$ take value in
  $\mathcal{X}\subseteq \real^n$ and $\newparamset\subseteq \real^d$,
  respectively.}

{We make the following assumptions.}
\change{There exists a norm $\|\cdot\|_{\mcX}$ on $\mcX$ and}
  \begin{enumerate}[label=\textup{(A\arabic*)}, leftmargin=0.9 cm,noitemsep]
    \setcounter{enumi}{\value{saveenum}}
  \item \label{ass:1_}  there exists $c > 0$ such that for all $\newparam$, the map $ x \mapsto F(x,\newparam)$ is strongly infinitesimally contracting with respect to $\|\cdot\|_{\mcX}$ with rate $c$, i.e., $\osL_x{(F)} \leq{-}c$,
  \item \label{ass:2_}
   \change{there exists a norm $\|\cdot\|_{\newparamset}$ on $\newparamset$, and }$\ell_{\newparam} \geq 0$ such that for all $x$, the map $\newparam \mapsto F(x,\newparam)$ is Lipschitz from $(\newparamset, \|\cdot\|_{\newparamset})$ to $(\mcX,\|\cdot\|_{\mcX})$ with constant $\ell_{\newparam}$.
  \end{enumerate}

\change{
Assumption~\ref{ass:1_} implies that, for each $\newparam \in \newparamset$, there exists a unique $x_{\newparam}^\star \in \mathcal{X}$ satisfying $F(x_{\newparam}^\star,{\newparam}) = \fzero_n$. Then, we can define the map $\map{\xstar}{\newparamset}{\mcX}$ given by $\xstar(\newparam) = x_{\newparam}^\star$. Lemma~\ref{lemma:parametrized-contractions} in Appendix~\ref{app:1:Proofs and Additional Results} shows that $\xstar(\cdot)$ is Lipschitz from $(\Theta, \|\cdot\|_{\Theta})$ to $(\mcX, \|\cdot\|_{\mcX})$ with constant $\ell_{\theta}/c$.
With this set up in mind, in the following we define the \emph{time-varying equilibrium curve} which is key in our equilibrium tracking results.
\begin{defn}[Time-varying equilibrium curve]
Consider a continuously differentiable curve $\map{\newparam}{\realnonnegative}{\newparamset}$ and the system~\eqref{eq:system_1} satisfying Assumptions~\ref{ass:1_} and~\ref{ass:2_}. The \emph{time-varying equilibrium curve} is the map $t \mapsto x^\star(\newparam(t))$.
\end{defn}
}

Since $\xstar(\cdot)$ is Lipschitz, 
\oldchange{the curve} $\xstar(\newparam(\cdot))$ is \change{locally Lipschitz}
(see Lemma~\ref{lem:bound_xstar_dot} in Appendix~\ref{app:1:Proofs and Additional Results} for details). 
\oldchange{Additionally, }this curve satisfies $F(\xstar(\newparam(t)), \newparam(t)) = \fzero_n$ for all $t \geq 0$. In the following theorem, we provide tracking error bounds between any trajectory of~\eqref{eq:system_1} and the time-varying equilibrium curve.

\begin{theorem}[Equilibrium tracking for contracting dynamics]
\label{thm:bound_time_varying_eq}
Let $\map{\newparam}{\realnonnegative}{\newparamset}$ be continuously differentiable and
consider the dynamics~\eqref{eq:system_1}
satisfying \oldchange{Assumptions}~\ref{ass:1_} and~\ref{ass:2_}.  Let
$x^\star(\newparam(\cdot))$ be the \emph{time-varying equilibrium curve}
of~\eqref{eq:system_1}. Then, for any initial conditions ${x(0) \in \R^n}$,
${\newparam(0) \in \R^d}$ and \change{for all $t \geq 0$:}
\begin{enumerate}
\item\label{fact:eqtrack:3} 
\change{the tracking error $\|x(t) - \xstar(\newparam(t))\|_{\mcX}$ satisfies}
  \begin{multline*}
    \|x(t){-}\xstar(\newparam(t))\|_{\mcX}  \\ \leq \e^{-ct}\|x(0){-}\xstar(\newparam(0))\|_{\mcX} + \frac{\ell_{\newparam}}{c}\int_0^t \! \e^{-c(t-\tau)}\|\dot \newparam(\change{\tau})\|_{\newparamset}d\tau;
  \end{multline*}
\item\label{fact:eqtrack:3.5} 
\change{the residual $\|F(x(t),\theta(t))\|_{\mcX}$ satisfies}
  \begin{multline*}
    \change{\|F(x(t),\theta(t))\|_{\mcX}} \\ \change{\leq \e^{-ct}\|F(x(0),\theta(0))\|_{\mcX} + \ell_{\newparam} \int_0^t \! \e^{-c(t-\tau)}\|\dot \newparam(\tau)\|_{\newparamset}d\tau;}
  \end{multline*}
\item\label{fact:eqtrack:4}  
\change{the following asymptotic bounds hold:}
  \begin{align*}
    \ds\limsup_{t \to \infty}\|x(t){-}\xstar(\change{\theta}(t))\|_{\mcX} &\leq \frac{\ell_{\oldchange{\theta}}}{c^2} \limsup_{t \to \infty} \|\oldchange{\dot{\theta}}(t)\|_{\change{\Theta}}, \\
    \change{\ds\limsup_{t \to \infty}\|F(x(t),\theta(t))\|_{\mcX}} &\change{\leq \frac{\ell_{\theta}}{c} \limsup_{t \to \infty} \|\dot{\theta}(t)\|_{\change{\Theta}}.}
  \end{align*}
\end{enumerate}
\end{theorem}\smallskip
\begin{proof}
  \change{To prove item~\ref{fact:eqtrack:3}}, consider the auxiliary
  dynamics
  \begin{equation}
    \label{eq:system_3}
    \dot x(t) = F(x(t),\newparam(t)) + v(t) := T(x(t),\newparam(t),v(t)),
  \end{equation}
  where $\map{T}{\real^n \times \real^d \times \real^n}{\real^n}$ and
  $\map{v}{\realnonnegative}{\mathcal{X}}$. Note that by Assumption~\ref{ass:1_}, for fixed $\newparam$ and $v$, the map $x \mapsto T(x,\newparam,v)$ is
  strongly infinitesimally contracting with rate ${c > 0}$. Moreover, at
  fixed $x$, $\newparam$, the map $v \mapsto T(x,\newparam,v)$ is Lipschitz on $(\mcX,\|\cdot\|_{\mcX})$ with constant
  $\ell_v = 1$.  Consider the inputs $v_1(t) = \0_n$ and $v_2(t) =
  \xstardot(\newparam(t))$ and note that $\xstardot(\newparam(t)) = F(\xstar(\newparam(t)), \newparam(t)) +
  \xstardot(\newparam(t))$ so that the curve $\xstar(\newparam(\cdot))$ is a solution to
  the dynamical system~\eqref{eq:system_3} with input $v_2(t)$ and initial
  condition $\xstar(\newparam(0))$. Additionally, for any initial condition $x(0)
  \in \mcX$, the solution $x(t)$ to the dynamics~\eqref{eq:system_1} is a
  solution to the system~\eqref{eq:system_3} with input $v_1(t)$. By an
  application of the incremental ISS theorem for contracting dynamical systems~\cite[Theorem~3.16]{FB:24-CTDS} to the trajectories $x(\cdot),
  \xstar(\newparam(\cdot))$ arising from inputs $v_1(\cdot),v_2(\cdot)$, we have
  the bound for a.e. $t$
  \begin{align*}
    D^{+}\|x(t){-}\xstar(\newparam(t))\|_{\mcX} &\leq  {-}c\|x(t){-}\xstar(\newparam(t))\|_{\mcX} {+}\|\dot x^\star(\change{\theta}(t))\|_{\mcX}\\
    &\leq {-}c\|x(t){-}\xstar(\newparam(t))\|_{\mcX} {+} \frac{\ell_{\newparam}}{c}\|\dot \newparam(t)\|_{\newparamset},
  \end{align*}
  where the last inequality follows from Lemma~\ref{lem:bound_xstar_dot}.
  \change{Item~\ref{fact:eqtrack:3} is then a 
  consequence of the \GB inequality for Dini derivatives, e.g.~\cite[Lemma~11]{AD-SJ-FB:20o}.}  
  \change{To prove item~\ref{fact:eqtrack:3.5}, consider a trajectory $x(t)$ of~\eqref{eq:system_1} and let $V(t) = \|F(x(t),\theta(t))\|_{\mcX}$. Then, omitting dependencies of $x$ and $\theta$ on time, we compute}
  \begin{align*}
      &\change{D^+V(t) \overset{(\star)}{=} \lim_{h \to 0^+} \frac{\|F(x,\theta) + h\frac{d}{dt}F(x,\theta)\|_{\mcX} - \|F(x,\theta)\|_{\mcX}}{h}} \\
      &\change{\overset{(\triangle)}{\leq} \lim_{h \to 0^+} \frac{\|F(x,\theta) + h\jacx{F(x,\theta)F(x,\theta)}\|_{\mcX} - \|F(x,\theta)\|_{\mcX}}{h}} \\
      & \quad \change{+ \|\jacu{F(x,\theta)\dot{\theta}\|}_{\mcX}} \\
      &\change{\overset{\ref{ass:2_}}{\leq} \|F(x,\theta)\|_{\mcX}\lim_{h \to 0^+} \frac{\|I_n+ h\jacx{F(x,\theta})\|_{\mcX} - 1}{h} + \ell_\theta \|\dot{\theta}\|_{\Theta}} \\
      &\change{\leq \mu_{\mcX}(\jacx{F(x,\theta)}V(t) + \ell_\theta \|\dot{\theta}\|_{\Theta} \overset{\ref{ass:1_}}{\leq} -cV(t) + \ell_\theta \|\dot{\theta}\|_{\Theta},}
  \end{align*}
  \change{where $(\star)$ holds by a Taylor expansion of $F$ in $t$, inequality $(\triangle)$ is a consequence of $\frac{d}{dt}F(x,\theta) = \jacx{F(x,\theta)}\dot{x} + \jacu{F(x,\theta)}\dot{\theta}$ and the triangle inequality, inequalities~\ref{ass:2_} and~\ref{ass:1_} are a consequence of Assumptions~\ref{ass:2_} and~\ref{ass:1_}, respectively. Item~\ref{fact:eqtrack:3.5} then follows by
  the \GB inequality. }\change{Item~\ref{fact:eqtrack:4} is a consequence of items~\ref{fact:eqtrack:3} and~\ref{fact:eqtrack:3.5}.}
\end{proof}

\oldchange{Theorem~\ref{thm:bound_time_varying_eq} is a general result that establishes that one does not need to know $\dot{x}^\star(\newparam(t))$ in order to get an estimate on the tracking error. Indeed, by leveraging Lemma~\ref{lemma:parametrized-contractions}, we know that the map $\xstar$ has Lipschitz bound $\ell_\theta/c$ and this is one of the key steps in establishing the asymptotic bound~\ref{fact:eqtrack:4}. This bound gives designers insight on how they may speed up their dynamics to provide lower values of tracking error.}

\change{Additionally, if we have knowledge of $\dot{\newparam}$, we can augment the contracting dynamics~\eqref{eq:system_1} with a feedforward term that ensures an exponential decay to zero tracking error. To do so, consider a parameter-dependent vector field $\map{F}{\real^n \times \real^d}{\real^n}$ continuously differentiable in both arguments and satisfying  Assumption~\ref{ass:1_}. Let $\map{\newparam}{\realnonnegative}{\newparamset \subseteq \real^d}$ be continuously differentiable. We introduce the \emph{time-varying contracting dynamics with feedforward prediction}:
\begin{equation}
\label{eq:feedforward}
\begin{aligned}
\dot{x}(t) = F\bigl(&x(t),\newparam(t)\bigr) \\
&- \bigl(\jac{}_x F(x(t),\newparam(t))\bigr)^{-1}\jac{}_{\newparam} F(x(t),\newparam(t)) \dot{\newparam}(t).
\end{aligned}
\end{equation}}
\change{
  Note that Assumption~\ref{ass:1_} implies the inequality $\mu(\jac{}_xF(x,\theta)) \leq -c$, for all $x,\theta$.
  From~\cite[Theorem~2.9(ii)]{FB:24-CTDS}, we know that the eigenvalues of $\jac{}_xF(x,\theta)$ are in the open left half plane, which implies invertibility of $\jac{}_xF(x,\theta)$ and ensures that the dynamics~\eqref{eq:feedforward} are well-posed.
}

\change{
In the following result, we show that considering the dynamics~\eqref{eq:feedforward} we obtain exponential decay to zero tracking error.}
\oldchange{\begin{theorem}[Exact tracking with feedforward prediction]\label{thm:exact=tracking}
Let $\map{F}{\real^n \times \real^d}{\real^n}$ be a parameter-dependent vector field, and let $\map{\newparam}{\realnonnegative}{\newparamset \subseteq \real^d}$ be continuously differentiable. Assume $F$ is continuously differentiable in both arguments and \change{satisfies} Assumption~\ref{ass:1_}. \change{Consider the dynamics~\eqref{eq:feedforward}.} Then for all $t \geq 0,$
\begin{enumerate}
        \item\label{item:exact-1} the residual $\|F(x(t),\newparam(t))\|_{\mcX}$ satisfies
        $$
            \|F(x(t),\newparam(t))\|_{\mcX} \leq \e^{-ct}\|F(x(0),\newparam(0))\|_{\mcX};
        $$
        \item\label{item:exact-2} the tracking error $\|x(t)-\xstar(\newparam(t))\|_\mcX$ satisfies 
        $$
            \|x(t) - \xstar(\newparam(t))\|_{\mcX} \leq \frac{1}{c}\e^{-ct}\|F(x(0),\newparam(0))\|_{\mcX};
        $$
        {Additionally, if} $F$ is Lipschitz in its first argument with constant $\ell_x$ uniformly in $\theta$, then
        $$
          \|x(t) - \xstar(\theta(t))\|_\mcX \leq \frac{\ell_x}{c}\e^{-ct}\|x(0) - \xstar(\theta(0))\|_\mcX.
        $$
    \end{enumerate}
\end{theorem}
}\smallskip
\begin{proof}
    {To prove item~\ref{item:exact-1}, consider } \oldchange{a trajectory $x(t)$ of~\eqref{eq:feedforward}, and let $V(t) = \|F(x(t),\newparam(t))\|_{\mcX}$. Then, omitting dependencies of $x$ and $\newparam$ on time, we compute
    \begin{align*}
        &D^+ V(t)
        \overset{(\star)}{=} \lim_{h\to 0^+} \frac{\|F(x,\newparam) + h\frac{d}{dt}F(x,\newparam)\|_{\mcX} - \|F(x,\newparam)\|_{\mcX}}{h} \\
        &\overset{\eqref{eq:feedforward}}{=} \lim_{h\to 0^+} \frac{\|F(x,\newparam)+h\jac{}_x F(x,\newparam) F(x,\newparam)\|_{\mcX} - \|F(x,\newparam)\|_{\mcX}}{h} \\
        &\leq \|F(x,\newparam)\|_{\mcX}\lim_{h \to 0^+} \frac{\|I_n + h \jac{}_x F(x,\theta)\|_{\mcX}-1}{h} \\
        &\leq \mu_{\mcX}(\jac{}_x F(x,\newparam)) V(t) \leq -cV(t),
    \end{align*}
    where \change{$(\star)$} is by a Taylor expansion of $F$ in $t$
    and the \change{next equality} holds since~\eqref{eq:feedforward}
    implies that $\frac{d}{dt}F(x,\newparam) =
    \jac{}_xF(x,\newparam)\dot{x} + \jac{}_{\newparam}
    F(x,\newparam)\dot{\newparam} = \jac{}_x F(x,\newparam)F(x,\newparam).$
    The \GB inequality for Dini derivatives implies
    item~\ref{item:exact-1}. Item~\ref{item:exact-2} is a consequence of
    the fact that $\|F(x,\newparam)-F(\xstar(\newparam),\newparam)\|_{\mcX}
    \geq c\|x-\xstar(\newparam)\|_{\mcX}$ since
    $F(\xstar(\newparam),\newparam) = \vectorzeros[n]$ and for fixed
    $\newparam$, the map $x \mapsto F(x,\newparam)$ is invertible and the
    inverse map is Lipschitz on $(\mcX, \|\cdot\|_{\mcX})$ with constant
    $1/c$~\cite[Lemma~3.5]{FB:24-CTDS}.}
\end{proof}
\oldchange{Note that compared to Theorem~\ref{thm:bound_time_varying_eq}, Theorem~\ref{thm:exact=tracking} does not require Assumption~\ref{ass:2_} but does additionally require differentiability of $F$.
By assuming knowledge of $\dot{\newparam}$, we can leverage this information to exponentially achieve zero tracking error for arbitrary contracting dynamics, $F$. Comparatively, in Theorem~\ref{thm:bound_time_varying_eq}, we do not assume knowledge of $\dot{\newparam}$ and cannot expect to achieve zero tracking error as a result.
}

\section{Contracting Dynamics for Canonical \\ Convex Optimization Problems}

In this section, we provide a transcription from three canonical optimization problems to continuous-time dynamical systems which are strongly infinitesimally contracting. This transcription allows us to apply \oldchange{Theorems~\ref{thm:bound_time_varying_eq} and~\ref{thm:exact=tracking}} to time-varying instances of these problems. Specifically, we analyze (i) monotone inclusions, (ii) linear equality constrained problems, and (iii) composite minimization problems.

\newcommand{\fFB}{\subscr{\OF}{FB}}

\subsection{Monotone Inclusions}

We consider the following general problem which has found many applications in convex optimization, see~\cite{EKR-WY:21}.
\begin{prob}\label{prob:splitting}
Let
$\map{\OF}{\real^n}{\real^n}$ be monotone and
${\map{g}{\real^n}{\realextended}}$ be convex. We are interested in solving
the \emph{monotone inclusion problem}
\begin{equation}\label{eq:mon-splitting}
\text{Find } \xstar \in \real^n \text{ s.t. } \quad \z{n} \in (\OF + \partial g)(\xstar).
\end{equation}
\end{prob}
We make the following assumptions on $\OF$ and $g$:
\begin{assumption}\label{assm:mon-splitting}
    $\map{\OF}{\real^n}{\real^n}$ is strongly monotone with 
    parameter $m$ and Lipschitz on $(\real^n,\|\cdot\|_2)$ with constant $\lf$. The map
    $\map{g}{\real^n}{\realextended}$ is CCP. \oprocend
\end{assumption}

Under Assumption~\ref{assm:mon-splitting}, the monotone inclusion problem~\eqref{eq:mon-splitting} has a unique solution due to strong monotonicity of $\OF$.

Monotone inclusion problems of the form~\eqref{eq:mon-splitting} are prevalent in convex optimization and data science and we present two canonical problems which can be stated in terms of the monotone inclusion~\eqref{eq:mon-splitting}. 

\begin{example}[Convex minimization]
First, consider the convex optimization problem
\begin{equation}\label{eq:constrained}
\min_{x \in \real^n} f(x) + g(x),
\end{equation}
where $\map{f}{\real^n}{\real}$ is strongly convex and continuously differentiable and $\map{g}{\real^n}{\realextended}$ is CCP. In this case, the unique point $x^* \in \real^n$ that minimizes~\eqref{eq:constrained} also solves the inclusion problem~\eqref{eq:mon-splitting} with $\OF = \nabla f$. 
\end{example}

\begin{example}[Variational inequalities]
Second, consider the \emph{variational inequality} defined by the continuous monotone mapping ${\map{\OF}{\real^n}{\real^n}}$ and nonempty, convex, and closed set $\mcC$ which is the problem
\begin{equation}\label{eq:var-ineq}
\text{Find } x^\star \in \mcC \text{ s.t. } \quad \OF(x^\star)^\top (x-x^\star) \geq 0, \quad \forall x \in \mcC.
\end{equation}
We denote problem~\eqref{eq:var-ineq} by $\VI{\OF}{\mcC}$. 
It is known that ${x^* \in \mcC}$ solves $\VI{\OF}{\mcC}$ if and only if for all $\gamma > 0$, $x^*$ is a fixed point of the map $\proj_{\mcC} \circ (\Id - \gamma \OF)$, i.e.,
$x^* = \proj_{\mcC}(x^* - \gamma \OF({x^*})).$
In turn, this fixed-point condition is equivalent to asking $x^*$ to solve the monotone inclusion problem~\eqref{eq:mon-splitting} with $g = \iota_{\mcC}$, see, e.g.,~\cite[pp.~37]{EKR-SB:16}. Variational inequalities of the form~\eqref{eq:var-ineq} have found applications in computing Nash and Wardrop equilibria in games~\cite{DP-BG-FP-MK-JL:19}.
\end{example}

To solve the monotone inclusion problem~\eqref{eq:mon-splitting}, we
\change{consider} the following dynamics \change{from~\cite{BA-HA:15}}, \change{called} the \emph{continuous-time
forward-backward splitting dynamics} with parameter $\gamma > 0$:
\begin{equation}\label{eq:fwd-backward-dynamics}
\dot{x} = -x + \prox{\gamma g}(x - \gamma \OF(x)) =: \fFB^\gamma(x).
\end{equation}
\oldchange{The name forward-backward splitting dynamics comes from the classical forward-backward splitting algorithm from monotone operator theory, see, e.g.,~\cite[Section~26.5]{HHB-PLC:17}.}
\begin{remark}
    When $g = \iota_{\mcC}$ for some convex and closed set $\mcC$, then for any $\gamma > 0$, $\prox{\gamma g} = \proj_{\mcC}$, we are solving $\VI{\OF}{\mcC}$
    ~\eqref{eq:var-ineq} and the dynamics~\eqref{eq:fwd-backward-dynamics} are the \oldchange{\emph{projected dynamics}}
    $$
        \dot{x} = -x + \proj_{\mcC}(x - \gamma \OF(x)),
    $$
    which were studied in~\cite{XBG:03}. 
    Alternatively, when $\OF = \nabla f$ for some continuously differentiable convex function $f$, we are solving the convex optimization problem~\eqref{eq:constrained} and the dynamics~\eqref{eq:fwd-backward-dynamics} correspond to the \oldchange{\emph{proximal gradient dynamics}}
    \begin{equation}\label{eq:prox-grad-dynamics}
        \dot{x} = -x + \prox{\gamma g}(x- \gamma \nabla f(x)),
    \end{equation}
    which were studied in~\cite{SHM-MRJ:21}. \oprocend
\end{remark}

First, we establish that equilibrium points of the dynamics~\eqref{eq:fwd-backward-dynamics} correspond to solutions of the inclusion problem~\eqref{eq:mon-splitting}.

\begin{prop}[Equilibria of~\eqref{eq:fwd-backward-dynamics}]\label{prop:fwd-bwd-properties}
  Suppose Assumption~\ref{assm:mon-splitting} holds. Then for any $\gamma > 0$, 
      $\z{n} \in (\OF + \partial g)(x^*)$ if and only if $x^*{\in \real^n}$ is an equilibrium point of the dynamics~\eqref{eq:fwd-backward-dynamics}.
\end{prop}
\begin{proof}
    Note that equilibria of~\eqref{eq:fwd-backward-dynamics}, $x^* \in \real^n$, satisfy the fixed point equation
    \begin{equation}\label{eq:fwd-bwd-fixedpoint}
        x^* = \prox{\gamma g}(x^* - \gamma \OF(x^*)).
    \end{equation}
    Moreover, it is known that fixed points of the form~\eqref{eq:fwd-bwd-fixedpoint} also solve the monotone inclusion $\z{n} \in (\OF + \partial g)(x^*)$, see, e.g.,~\cite[Proposition~26.1(iv)(a)]{HHB-PLC:17}, noting that $\prox{\gamma g}$ is the resolvent of $\partial g$ with parameter $\gamma$.
\end{proof}

\begin{remark}
    Proposition~\ref{prop:fwd-bwd-properties} continues to hold under the assumption of monotone $\OF$~\cite[Propositon~26.1(iv)(a)]{HHB-PLC:17}. \oprocend 
\end{remark}

Next, we establish that the dynamics~\eqref{eq:fwd-backward-dynamics} are contracting under assumptions on the parameter $\gamma$.

\begin{theorem}[Contractivity of~\eqref{eq:fwd-backward-dynamics}]\label{thm:contractivity-fwd-bwd}
  \label{thm:contractivity-forwardbackward}
  Suppose Assumption~\ref{assm:mon-splitting} holds. Then
  \begin{enumerate}
  \item\label{item:fwd-bwd-contracting} for every $\gamma \in {]0,2m/\ell^2[}$, the
    dynamics~\eqref{eq:fwd-backward-dynamics} are strongly infinitesimally contracting with respect
    to $\norm{\cdot}{2}$ with rate $1-\sqrt{1 - 2\gamma m + \gamma^2
      \ell^2}$. Moreover,
  the contraction rate is optimized at $\gamma^* = m/\ell^2$.
    \setcounter{saveenum}{\value{enumi}}
  \end{enumerate}
  Additionally,
  \begin{enumerate}\setcounter{enumi}{\value{saveenum}}
  \item\label{item:proximal-gradient} if $\OF = \nabla f$ for some strongly convex $\map{f}{\real^n}{\real}$, for every $\gamma \in {]0,2/\ell[}$, the dynamics~\eqref{eq:fwd-backward-dynamics} are strongly infinitesimally contracting with respect to $\norm{\cdot}{2}$ with rate $1-\max\{|1-\gamma m|, |1-\gamma \ell|\}$. Moreover, the contraction rate is optimized at $\gamma^* = 2/(m+\ell)$;
  \item\label{item:affine-fwd-bwd} if $\OF(x) = Ax + b$ for all $x \in \real^n$, with ${A = A^\top \succ 0}$, then for every $\gamma \in {]1/\lambda_{\min}(A), +\infty[}$, the dynamics~\eqref{eq:fwd-backward-dynamics} are strongly infinitesimally contracting with respect to the norm ${\|\cdot\|_{(\gamma A-I_n)}}$ with rate $1$.
  \end{enumerate}
\end{theorem}\smallskip

\begin{proof}
    Regarding item~\ref{item:fwd-bwd-contracting} note that since
    $\map{g}{\real^n}{\realextended}$ is CCP, for every $\gamma > 0$,
    $\prox{\gamma g}$ is a nonexpansive map with respect to the
    $\norm{\cdot}{2}$ norm~\cite[Proposition~12.28]{HHB-PLC:17}. Moreover,
    for every $\gamma > 0$ the map $\Id - \gamma \OF$ has Lipschitz
    constant upper bounded by $\sqrt{1 - 2\gamma m + \gamma^2\ell^2}$ with
    respect to the $\norm{\cdot}{2}$ norm~\cite[pp. 16]{EKR-SB:16}. Since
    the Lipschitz constant of the composition of the two maps is upper
    bounded by the product of the Lipschitz constants, in light of
    Lemma~\ref{lemma:discrete-to-continuous-contraction} in
    Appendix~\ref{app:1:Proofs and Additional Results}, we conclude that
    for $\gamma \in {]0,2m/\ell^2[}$, $\osL(\fFB^\gamma) \leq -1 +
    \Lip(\prox{\gamma g} \circ (\Id - \gamma \OF)) \leq -1 + \sqrt{1 -
      2\gamma m + \gamma^2\ell^2} < 0.$ Moreover, minimizing
    $\osL(\fFB^\gamma)$ corresponds to minimizing $1 - 2\gamma m +
    \gamma^2\ell^2$ as a function of $\gamma \in {]0, 2m/\ell^2[}$ {. This
      minimization occurs} at $\gamma^* = m/\ell^2$ and yields a one-sided
    Lipschitz estimate of $\osL(\fFB^{\gamma^*}) \leq -1 + \sqrt{1 -
      m^2/\ell^2} < 0$.

    Item~\ref{item:proximal-gradient} follows the same argument as in item~\ref{item:fwd-bwd-contracting} where instead one shows that for all $\gamma > 0$, $\Lip(\Id - \gamma \nabla f) \leq \max\{|1-\gamma m|, |1-\gamma \ell|\}$ as in, e.g.,~\cite[pp.~15]{EKR-SB:16}. Then for all $\gamma \in {]0,2/\ell[}$, $\osL(\fFB^{\gamma}) \leq -1 + \max\{|1-\gamma m|, |1-\gamma \ell|\} < 0$. Moreover, the optimal choice of $\gamma$ is $\gamma^* = 2/(m + \ell)$ and the corresponding bound on $\osL(\fFB^{\gamma^*})$ is $-1 + (\kappa - 1)/(\kappa + 1)$, where $\kappa := \ell/m \geq 1$~\cite[pp.~15]{EKR-SB:16}.

    Regarding item~\ref{item:affine-fwd-bwd}, we compute the Jacobian of $\fFB^\gamma$ for all ${x \in \real^n}$ for which it exists, i.e., $\jac{\fFB^{\gamma}}(x) = -I_n + \jac{\prox{\gamma g}}(I_n - \gamma (Ax + b)) (I_n - \gamma A)$. Note that for all $x \in \real^n$ for which the Jacobian exists, 
    there exists $G = G^\top \in \real^{n\times n}$ with $0 \preceq G \preceq I_n$ satisfying $\jac{\prox{\gamma g}}(I_n - \gamma (Ax + b)) = G,$ see Lemma~\ref{lemma:symmetry-prox-Moreau} in Appendix~\ref{app:1:Proofs and Additional Results}. \oldchange{Additionally, we recall that, for any matrix $A$,  the log-norm translation property holds. That is, $\mu(A + cI_n) = \mu(A) + c$, for all $c \in \real$.} Then for any norm,
    \begin{equation}
    \label{ineq:sup_lognorm_DF}
    \sup_{x} \mu(\jac{\fFB^{\gamma}}(x)) \leq -1 + \max_{0 \preceq G \preceq I_n} \mu(G(I_n-\gamma A)),
    \end{equation}
    where the $\sup$ is over all $x$ for which $\jac{\fFB^{\gamma}}(x)$ exists.
    Moreover, for $\gamma > 1/\lambda_{\min}(A)$, $\gamma A - I_n$ is positive definite and $G(I_n - \gamma A) = (-G)(\gamma A - I_n)$ is the product of two symmetric matrices. An application of Sylvester's law of inertia implies that $G(I_n - \gamma A)$ has all real eigenvalues and that it has the same number of positive, zero, and negative eigenvalues as $-G$ does, i.e., all eigenvalues are nonpositive. Then from~\cite[Lemma~2]{VC-AG-AD-GR-FB:23c}, with the choice of norm $\|\cdot\|_{(\gamma A - I_n)}$, we find 
    \begin{multline*}
        \mu_{(\gamma A - I_n)}(G(I_n - \gamma A)) = \mu_{(\gamma A - I_n)}((-G)(\gamma A - I_n)) \\
        = \max\setdef{\realpart(\lambda)}{\lambda \text{ is an eigenvalue of } G(I_n - \gamma A)} \leq 0,
    \end{multline*}
    \oldchange{where in the last equality we used the definition of $\ell_2$ log-norm}.
    Since this equality holds for all symmetric $G$ satisfying ${0 \preceq G \preceq I_n}$, by applying inequality~\eqref{ineq:sup_lognorm_DF} we have
    $$\sup_{x \in \real^n} \mu_{(\gamma A - I_n)}(\jac{\fFB^{\gamma}}(x)) \leq -1.$$
    This inequality proves the result.
\end{proof}
\begin{remark}
\oldchange{
The rates of contraction in Theorem~\ref{thm:contractivity-fwd-bwd}\ref{item:fwd-bwd-contracting} and~\ref{item:proximal-gradient} are essentially consequences of the standard contraction rates of the forward-backward splitting algorithm in monotone operator theory, see~\cite[pp.~25]{EKR-SB:16} and~\cite[Proposition~26.16]{HHB-PLC:17}. In contrast, Theorem~\ref{thm:contractivity-fwd-bwd}\ref{item:affine-fwd-bwd} provides an improved and sharp rate of contraction in the case of affine $\OF$ for an increased range of $\gamma$. Note that this rate cannot be improved. To see this fact, consider $g = \iota_{\{b\}}$ so that the dynamics are $\dot{x} = -x + b$, which are contracting with rate equal to $1$. It is an open question whether the contraction rates in Theorem~\ref{thm:contractivity-fwd-bwd}\ref{item:fwd-bwd-contracting} and~\ref{item:proximal-gradient} can be improved to $1$ with different choice of $\gamma$ \change{and norm}.
} \oprocend
\end{remark}
\paragraph*{\change{Parameter-varying case}}
\oldchange{
Consider the \emph{parameter-varying inclusion problem}
\begin{equation}\label{eq:parameter-monotone-inc}
   \text{For } \theta \in \Theta, \text{ find } \xstar(\theta) \in \real^n \text{ s.t. } \vectorzeros[n] \in (\OF_\theta + \partial g_\theta)(\xstar(\theta)),
\end{equation}
where for each $\theta \in \Theta,$ the map $\map{\OF_\theta}{\real^n}{\real^n}$ is Lipschitz and strongly monotone and the map $\map{g_\theta}{\real^n}{\realextended}$ is CCP. 
Then, for suitable $\gamma > 0$, the corresponding parameter-varying forward-backward splitting dynamics are given by
\begin{equation}\label{eq:parametric-fwd-bwd-dynamics}
    \dot{x} = -x + \prox{\gamma g_\theta}(x - \gamma \OF_\theta(x)) =: \fFB^\gamma(x,\theta).
\end{equation}
For each $\theta \in \Theta$, these dynamics are strongly contracting in
view of Theorem~\ref{thm:contractivity-fwd-bwd} and thus the
problem~\eqref{eq:parameter-monotone-inc} has a unique solution
$\xstar(\theta)$. Moreover, when $\map{\theta}{\realnonnegative}{\Theta}$
is a continuously differentiable curve, under a Lipschitz condition on the
map $\theta \mapsto \fFB^\gamma(x,\theta),$
Theorem~\ref{thm:bound_time_varying_eq} ensures that trajectories
of~\eqref{eq:parametric-fwd-bwd-dynamics} track $\xstar(\theta(t))$ with a
tracking error proportional to $\|\dot{\theta}(t)\|_{\Theta}$ after a
transient. Such a Lipschitz condition holds if, e.g., the maps $\theta
\mapsto \prox{\gamma g_{\theta}}(x)$ and $\theta \mapsto \OF(x,\theta)$ are
Lipschitz uniformly in $x$\footnote{\oldchange{To determine whether the map
  $\theta \mapsto \prox{\gamma g_\theta}(x)$ is Lipschitz, one can employ
  sensitivity analysis of parametric programs. We refer the interested
  reader to~\cite{AVF:76} for sufficient conditions.}}. } \oldchange{If,
  additionally, $\fFB^\gamma$ is differentiable in both arguments, one can
  design a feedforward term to attain zero tracking error leveraging
  Theorem~\ref{thm:exact=tracking}. In Section~\ref{sec:CBF}, we consider
  an application that takes this approach.}

\subsection{Linear Equality Constrained Optimization}
We study another canonical problem in convex optimization.

\begin{prob}\label{prob:eq-constrained}
    Let $\map{f}{\real^n}{\real}$ be convex, $A \in \real^{m \times n}$, and $b \in \real^m$. Consider the equality-constrained problem
    \begin{equation}\label{eq:equality-constrained}
    \begin{aligned}
    \min_{x \in \real^n} \quad & f(x), \\
    \text{s.t.} \quad & Ax = b.
    \end{aligned}
    \end{equation}
\end{prob}

We make the following assumptions on $f$ and $A$: 
\begin{assumption}\label{assm:primaldual}
$\map{f}{\real^n}{\real}$ is continuously differentiable, strongly convex and strongly smooth with parameters $\mf$
and $\lf$, respectively.
The matrix $A \in \real^{m \times n}$ satisfies ${\amin I_m \preceq AA^\top \preceq \amax I_m}$ for $\amin,\amax \in \realpositive$. \oprocend
\end{assumption}

Note that Problem~\ref{prob:eq-constrained} is a special case of Problem~\ref{prob:splitting} with $\OF = \nabla f$ and $g = \iota_{\mcC}$ where $\mcC = \setdef{z \in \real^n}{Az = b}$. In this case we have that $\prox{\alpha g} = \proj_{\mcC}$ and $\proj_{\mathcal{C}}(z) = z - A^\dagger (Az - b) = (I_n - A^\dagger A)z + A^\dagger b$, where $A^\dagger$ denotes the pseudoinverse of $A$. In the context of the forward-backward splitting dynamics~\eqref{eq:fwd-backward-dynamics}, the dynamics read
\begin{equation}\label{eq:projected-lin-eq}
	\dot{x} = -x + (I_n - A^\dagger A)(x - \gamma\nabla f(x)) + A^\dagger b.
\end{equation}
In light of Theorem~\ref{thm:contractivity-fwd-bwd}\ref{item:proximal-gradient}, for $\gamma \in {]0, 2/\lf[}$, the dynamics~\eqref{eq:projected-lin-eq} are strongly infinitesimally contracting with respect to $\norm{\cdot}{2}$ with rate $1-\max\{|1-\gamma \mf|,|1-\gamma \lf|\}$.

The downside to using the dynamics~\eqref{eq:projected-lin-eq} is the cost of computing $A^\dagger$. To remedy this issue, a common approach is to leverage duality and jointly solve primal and dual problems. In what follows, we take this approach and study contractivity of the corresponding primal-dual dynamics.

The Lagrangian associated to the problem~\eqref{eq:equality-constrained} is
the map $\map{L}{\real^n \times \real^m}{\real}$ given by $L(x,\lambda) =
f(x) + \lambda^\top (Ax - b)$. \change{Computing the gradient descent of
  $L$ in $x$ and gradient ascent of $L$ in $\lambda$, the
  \emph{continuous-time primal-dual dynamics} (also called
  Arrow-Hurwicz-Uzawa flow~\cite{KJA-LH-HU:58}) are}
\begin{equation}\label{eq:primal-dual}
\begin{aligned}
\dot{x} &= -\nabla_x L(x,\lambda) = -\nabla f(x) - A^\top \lambda, \\
\dot{\lambda} &= \nabla_{\lambda} L(x,\lambda) = Ax - b.
\end{aligned}
\end{equation}

\begin{theorem}[Contractivity of primal-dual dynamics]\label{thm:primal-dual-contractivity}
    Suppose Assumption~\ref{assm:primaldual} holds. Then the continuous-time primal-dual dynamics~\eqref{eq:primal-dual} are strongly infinitesimally contracting with respect to $\|\cdot\|_{P}$ with rate $c > 0$ where
    \oldchange{
    \begin{align}
    P &= \begin{bmatrix}
      I_n & \alpha A^\top \\ \alpha A & 
      I_m \end{bmatrix} \succ 0, \ 
    \alpha=\frac{1}{2}\min\Big\{\frac{1}{\lf},\frac{\mf}{\amax}\Big\}, \
    \text{and}\label{eq:P-def} \\
    c&=\frac{1}{2}\alpha\amin = 
    \frac{1}{4}\min\Big\{\frac{\amin}{\lf},\frac{\amin}{\amax}\mf\Big\}.\label{eq:c-def}
    \end{align}}
\end{theorem}\smallskip
\begin{proof}
    Since $f$ is continuously differentiable, convex, and strongly smooth, it is almost everywhere twice differentiable so the Jacobian of the dynamics~\eqref{eq:primal-dual} exists almost everywhere and is given by
$\subscr{J}{PD}(z) := \begin{bmatrix}
-\Hess f(x) & -A^\top \\ A & 0 
\end{bmatrix},$
where $z = (x,\lambda) \in \real^{n+m}$.
To prove strong infinitesimal contraction it suffices to show that for all $z$ for which $\subscr{J}{PD}(z)$ exists, the bound $\mu_{P}(\subscr{J}{PD}(z)) \leq -c$ holds for $P,c$ given in~\eqref{eq:P-def} and~\eqref{eq:c-def}, respectively. The assumption of strong convexity and strong smoothness of $f$ further imply that $\mf I_n \preceq \Hess f(x) \preceq \lf I_n$ for all $x$ for which the Hessian exists. Moreover, it holds:
\begin{equation*}
    \sup_{z}\mu_{P}(\subscr{J}{PD}(z)) \leq \max_{\mf I_n \preceq B \preceq \lf I_n} \mu_{P}\left(\begin{bmatrix}
        -B & -A^\top \\ A & 0
    \end{bmatrix}\right), 
\end{equation*}
where the $\sup$ 
is over all points for which $\subscr{J}{PD}(z)$ exists. The result is then a consequence of Lemma~\ref{lemma:saddle-matrices} in Appendix~\ref{app:2:Logarithmic norm of Hurwitz saddle  matrices}.
\end{proof}

\begin{remark}
\label{rem:comparison}
Our method of proof in Lemma~\ref{lemma:saddle-matrices} follows the same method as was presented in~\cite[Lemma~2]{GQ-NL:19}, but uses a sharper upper bounding to yield a sharper contraction rate of \oldchange{$\ds\frac{1}{4}\min\Big\{\frac{\amin}{\lf},\frac{\amin}{\amax}\mf\Big\}$} compared to the estimate $\ds\frac{1}{8}\min\Big\{\frac{\amin}{\lf},\frac{\amin}{\amax}\mf\Big\}$ in~\cite[Lemma~2]{GQ-NL:19}. The sharper upper bounding is a consequence of an appropriate matrix factorization and a less conservative bounding of the square of a difference of matrices. \oprocend
\end{remark}

\paragraph*{\change{Parameter-varying case}}
\oldchange{
Consider the \emph{parameter-dependent equality-constrained minimization problem}:
\begin{equation}\label{eq:theta-equality-constrained}
    \begin{aligned}
    \min_{x \in \real^n} \quad & f_\theta(x), \\
    \text{s.t.} \quad & Ax = b_\theta,
    \end{aligned}
    \end{equation}
where for each $\theta \in \Theta$, $b_\theta \in \real^m$ and the map $f_\theta$ is continuously differentiable, strongly convex, and strongly smooth with parameters $\rho$ and $\ell$, respectively. We also assume that $A$ is full row rank.}

\oldchange{The \emph{parameter-varying primal-dual dynamics} are
\begin{equation}\label{eq:theta-primal-dual}
\begin{aligned}
\dot{x} &= -\nabla f_\theta(x) - A^\top \lambda, \\
\dot{\lambda} &= Ax - b_\theta,
\end{aligned}
\end{equation}}
\oldchange{and by Theorem~\ref{thm:primal-dual-contractivity}, for fixed $\theta$,~\eqref{eq:theta-primal-dual} is guaranteed to converge {to} the unique primal-dual pair solving~\eqref{eq:theta-equality-constrained}. Moreover, when $\map{\theta}{\realnonnegative}{\Theta}$ is a differentiable curve, under the assumptions that $\theta \mapsto \nabla f_\theta(x)$ and $\theta \mapsto b_\theta$ are Lipschitz, the dynamics~\eqref{eq:theta-primal-dual} are guaranteed to track $\xstar(\theta(t)),\lambda^\star(\theta(t))$ with a tracking error proportional to $\|\dot{\theta}(t)\|$ after a transient. Further, if $\nabla_\theta f$ is differentiable both in $x$ and $\theta$ and $b_\theta$ is differentiable in $\theta$, then we can design a feedforward term involving $\dot{\theta}$ leveraging Theorem~\ref{thm:exact=tracking} to attain zero tracking error.}

\oldchange{It is important to note that we have not let $A$ depend on the parameter $\theta$ since the norm with respect to which the dynamics~\eqref{eq:primal-dual} are contracting depends on $A$. If $A$ depends on $\theta$, then the norm with respect to which the dynamics are contracting is also parameter-dependent and the results from Theorems~\ref{thm:bound_time_varying_eq} and~\ref{thm:exact=tracking} do not directly apply.}

\subsection{Composite Minimization}

Finally, we study a composite minimization problem.
\begin{prob}\label{prob:general-two-objective}
Let $\map{f}{\real^n}{\real}$ and $\map{g}{\real^m}{\realextended}$ be convex, and $A \in \real^{m \times n}$. We consider the problem
    \begin{equation}\label{eq:two-objective-withA}
\min_{x \in \real^n} f(x) + g(Ax).
\end{equation}
\end{prob}

We make the following assumptions on $f,g$, and $A$: 
\begin{assumption}\label{assm:proximal-aug}
    $\map{f}{\real^n}{\real}$ is continuously differentiable, strongly convex with parameter $\mf$, and strongly smooth with parameter $\lf$. The map $\map{g}{\real^m}{\realextended}$ is CCP. Finally, $A \in \real^{m \times n}$ satisfies $\amin I_m \preceq AA^\top \preceq \amax I_m$ for $\amin,\amax \in \realpositive$. \oprocend
\end{assumption}
While the optimization problem~\eqref{eq:two-objective-withA} may appear to be a special case of~\eqref{eq:constrained}, it may be computationally challenging to compute the proximal operator of $g \circ A$ even if the proximal operator of $g$ may have a closed-form expression. Thus, we treat this problem separately. 

The optimization problem~\eqref{eq:two-objective-withA} is equivalent to

\begin{equation}\label{eq:rewritten-two-objective}
\begin{aligned}
\min_{x\in\real^n,y\in \real^m} & f(x) + g(y), \\
\text{s.t. } \quad & Ax - y = \vectorzeros[m].
\end{aligned}
\end{equation}

We leverage the proximal augmented Lagrangian approach proposed in~\cite{NKD-SZK-MRJ:19}. 
For $\gamma > 0$, define the \emph{augmented Lagrangian} associated to~\eqref{eq:rewritten-two-objective} $\map{L_\gamma}{\real^n \times \real^m \times \real^m}{\real}$ by
\begin{equation}
L_\gamma(x,y,\lambda) = f(x) + g(y) + \lambda^\top (Ax - y) + \frac{1}{2\gamma}\|Ax - y\|_2^2,
\end{equation}
and, by a slight abuse of notation, the \emph{proximal augmented Lagrangian} $\map{L_\gamma}{\real^n \times \real^m}{\real}$ by
\begin{equation}
\label{eq:proximal augmented Lagrangian}
L_\gamma(x,\lambda) = f(x) + M_{\gamma g}(Ax + \gamma \lambda) - \frac{\gamma}{2}\|\lambda\|_2^2.
\end{equation}
The proximal augmented Lagrangian \change{equals} the augmented Lagrangian where the minimization over $y$ has already been performed and the optimal value for $y$ has been substituted; see details in~\cite[Theorem~1]{NKD-SZK-MRJ:19}. Moreover, minimizing~\eqref{eq:two-objective-withA} corresponds to finding saddle points of~\eqref{eq:proximal augmented Lagrangian}. 
To this end, the 
\change{proximal augmented Lagrangian primal-dual dynamics are}
\begin{equation}
\label{eq:prox-aug-primal-dual}
\begin{aligned}
\dot{x} &= -\nabla_x L_\gamma(x,\lambda) = -\nabla f(x) {-} A^\top \nabla M_{\gamma g}(Ax + \gamma \lambda), \\
\dot{\lambda} &= \nabla_{\lambda} L_\gamma(x,\lambda) = \gamma(-\lambda + \nabla M_{\gamma g}(Ax + \gamma \lambda)).
\end{aligned}
\end{equation}

\newcommand{\proxineqname}{proximal inequality-constrained primal-dual dynamics\xspace}
\change{Before providing contraction estimates, we showcase the specific form of the dynamics~\eqref{eq:prox-aug-primal-dual} in the case of inequality-constrained minimization problems of the form $\min\setdef{f(x)}{Ax \leq b}$. In this case, $g = \iota_{\mcC}$, where $\mcC = \setdef{z \in \real^m}{z \leq b}$. Then $\prox{\gamma g}(z) = \proj_{\mcC}(z) = \min\{z,b\}$ and the corresponding gradient of the Moreau envelope is $\nabla M_{\gamma g}(z) = \frac{1}{\gamma}\relu(z - b)$, where the $\min$ and $\relu$ are applied entrywise. Finally, the dynamics~\eqref{eq:prox-aug-primal-dual} take the form}
\begin{align}
    \change{\dot{x}} &\change{ \;= -\nabla f(x) - \frac{1}{\gamma}A^\top \relu(Ax + \gamma \lambda - b)}, \nonumber \\
    \change{\dot{\lambda}} &\change{ \;= -\gamma\lambda + \relu(Ax + \gamma\lambda - b),\label{eq:prox-aug-inequality}}
\end{align}
\change{which we refer to as the \emph{\proxineqname}.}

\change{Next, we turn to contraction analysis of the dynamics~\eqref{eq:prox-aug-primal-dual}.}
To provide estimates on the contraction rate and the norm with respect to which the dynamics~\eqref{eq:prox-aug-primal-dual} are strongly infinitesimally contracting, we need to
introduce a useful nonlinear program.
For $\varepsilon \in {]0, 1/\sqrt{\amax}[}$, consider the nonlinear program
\begin{subequations}\label{eq:nonlinear-prog}
\begin{alignat}{2}
&\!\max_{c\geq 0,\alpha \geq 0,\varkappa \geq 0}        &\qquad& c\label{eq:optProb}\\
&\text{s.t.} &      & \alpha \leq \min\Big\{\frac{1}{\sqrt{\amax}} - \varepsilon, \frac{\gamma}{\amax}\Big\},\label{eq:alpha-constraint}\\
&                  &      & \varkappa \geq \frac{2}{3},\label{eq:constraint2} \\
&    &    & c \leq \Big(\frac{3}{4}-\frac{1}{2\varkappa}\Big)\alpha \amin, \label{eq:constraint3}\\
& & & h(c,\alpha,\varkappa) \geq 0, \label{eq:nonlinear-constraint}
\end{alignat}
\end{subequations}
with $\map{h}{\realnonnegative \times \realnonnegative \times \realnonnegative}{\real}$ given by
\begin{align*}
h(c,\alpha,\varkappa) =& \ 2\mf - \relu\Big(2\alpha-\frac{2}{\gamma}\Big)\amax - 2c \\&- \alpha\varkappa\frac{\amax}{\amin} \Big(\gamma^2 \frac{\amax}{\amin} + (\lf + \frac{\amax}{\gamma} + 2c)^2 \\  &+ 2\gamma\frac{\amax}{\amin} (\lf + \frac{\amax}{\gamma} + 2c)\Big). 
\end{align*} 
\begin{arxiv}
We prove in Lemma~\ref{lemma:existence-opt} in Appendix~\ref{app:3:Generalized Saddled Point Matrice} that there exist finite values $c > 0, \alpha > 0, \varkappa > 0$ that solve the problem~\eqref{eq:nonlinear-prog}. 
\end{arxiv}
\begin{theorem}[Contractivity of the dynamics~\eqref{eq:prox-aug-primal-dual}]\label{thm:contractivity-pal}
    Suppose Assumption~\ref{assm:proximal-aug} holds and let $\gamma > 0$ be arbitrary. Then the primal-dual dynamics~\eqref{eq:prox-aug-primal-dual} are strongly infinitesimally contracting with respect to $\|\cdot\|_{P}$ with rate $c^\star > 0$ where
    \begin{equation}\label{eq:P-def-pal}
        P = \begin{bmatrix}
            I_n & \alpha^\star A^\top \\ \alpha^\star A & I_m
        \end{bmatrix},
    \end{equation}
    and $\alpha^\star > 0, c^\star > 0$ are the arguments solving problem~\eqref{eq:nonlinear-prog}.
\end{theorem}\smallskip
\begin{proof}
    Let $z = (x,\lambda) \in \real^{n + m}$ and let $\map{\OF}{\real^{n+m}}{\real^{n+m}}$ corresponds to the vector field~\eqref{eq:prox-aug-primal-dual} for $\dot{z} = \OF(z)$. Let $y := Ax + \gamma \lambda$ and define $G(y) := \gamma \Hess M_{\gamma g}(y)$ where it exists. The Jacobian of $\OF$ is then
    $$
    \!\!\!\! \jac{\OF}(z) = \begin{bmatrix}
    -\Hess f(x) - \frac{1}{\gamma}A^\top G(y)A & \!\!- A^\top G(y) \\
    G(y)A & \!\! -\gamma (I_m - G(y))
    \end{bmatrix},
    $$
which exists for a.e. $z$. We then aim to show that $\mu_{P}(\jac{\OF}(z)) \leq -c$ for all $z$ for which $\jac{\OF}(z)$ exists. First we note that in light of Lemma~\ref{lemma:symmetry-prox-Moreau} in Appendix~\ref{app:1:Proofs and Additional Results},
\begin{multline*}
\!\!\!\!\!\! \sup_{z} \mu_{P}(\jac{\OF}(z)) {\leq}\!\!\! \max_{\substack{0 \preceq G \preceq I_m \\ \mf I_n \preceq B \preceq \lf I_n}} \!\!\!\!\!\!\mu_{P}\left(\begin{bmatrix}
        {-}B {-} \frac{1}{\gamma} A^\top GA &\!\!\!\! {-}A^\top G \\ GA & \!\!\!\! \gamma (G {-} I_m)
    \end{bmatrix}\right),
\end{multline*}
where the $\sup$ is over all $z$ for which $\jac{\OF}(z)$ exists. The result is then a consequence of Lemma~\ref{lemma:generalized-saddle} in Appendix~\ref{app:2:Logarithmic norm of Hurwitz saddle  matrices}.
\end{proof}

\begin{remark}
    To the best of our knowledge, the solution to the nonlinear program~\eqref{eq:nonlinear-prog} provides \change{the most general} test for the contractivity of the dynamics~\eqref{eq:prox-aug-primal-dual}. The original work~\cite[Theorem~3]{NKD-SZK-MRJ:19}, proves exponential convergence provided that $\gamma > \lf - \mf$. Instead we prove contraction, a stronger property, for all $\gamma > 0$. \change{We compare contraction and convergence rate estimates in Figure~\ref{fig:contraction-rates} in Appendix~\ref{app:2:Logarithmic norm of Hurwitz saddle  matrices}.}
    \oprocend
\end{remark}

Note that any triple $(c,\alpha,\varkappa) \in \realnonnegative^3$ satisfying the constraints~\eqref{eq:alpha-constraint}-\eqref{eq:nonlinear-constraint} provides a suboptimal contraction estimate, i.e., the dynamics~\eqref{eq:prox-aug-primal-dual} are strongly infinitesimally contracting with rate $c$ (weakly contracting if $c = 0$) with respect to norm $\|\cdot\|_{P}$, where $\ds P = \begin{bmatrix}
        I_n & \alpha A^\top \\ \alpha A & I_m
    \end{bmatrix}$.
In what follows, we present computational considerations for estimating the optimal parameters $c^\star,\alpha^\star$ in the nonlinear program~\eqref{eq:nonlinear-prog}. 
Let $\varkappa > 2/3$ be fixed (e.g., at a value of $1$). Then we have the following bounds on $c$ which we will bisect on:
$$0 \leq c \leq c_{\max}^{\varkappa} := \min\Big\{\mf, \Big(\frac{3}{4}-\frac{1}{2\varkappa}\Big)\amin\alpha_{\max}\Big\},$$
where $\alpha_{\max} = \min\Big\{\frac{1}{\sqrt{\amax}} - \varepsilon, \frac{\gamma}{\amax}\Big\}$.
For any value of $c \in [0,c_{\max}^{\varkappa}]$, we check whether the following linear program (LP) is feasible:
\begin{subequations}\label{eq:alpha-feasibility}
\begin{alignat}{2}
&\!\text{Find}        &\qquad& \alpha\label{eq:feasProb}\\
&\text{s.t.} &      & 0 \leq \alpha \leq \alpha_{\max},\\
& & & c \leq \Big(\frac{3}{4}-\frac{1}{2\varkappa}\Big)\alpha\amin, \\
& & & h(c,\alpha,\varkappa) \geq 0.
\end{alignat}
\end{subequations}

Although $h(c,\alpha,\varkappa)$ is not linear in $\alpha$, the problem can be transformed into an equivalent LP since $\relu$ is piecewise linear.
\change{The} LP is feasible for $c = 0$ (with $\alpha = 0$). If the LP is feasible for $c = c_{\max}^{\varkappa}$, then $c_{\max}^{\varkappa}$ is the optimal contraction rate for this choice of $\varkappa$. More typically, the LP will not be feasible for $c_{\max}^{\varkappa}$ at which point we bisect on $c$, checking for feasible $\alpha$ for~\eqref{eq:alpha-feasibility} at the prescribed value of $c$ until we find a $\delta$-optimal value of $c$ with corresponding $\alpha$ that is feasible.

Optimizing further over $\varkappa$ can be done numerically either using nonlinear programming solvers or by using a grid search over $\varkappa$ and then the bisecting on $c$.

\begin{figure*}[th!]
	\centering
	\begin{tabular}{cc}
		\includegraphics[width=0.48\linewidth]{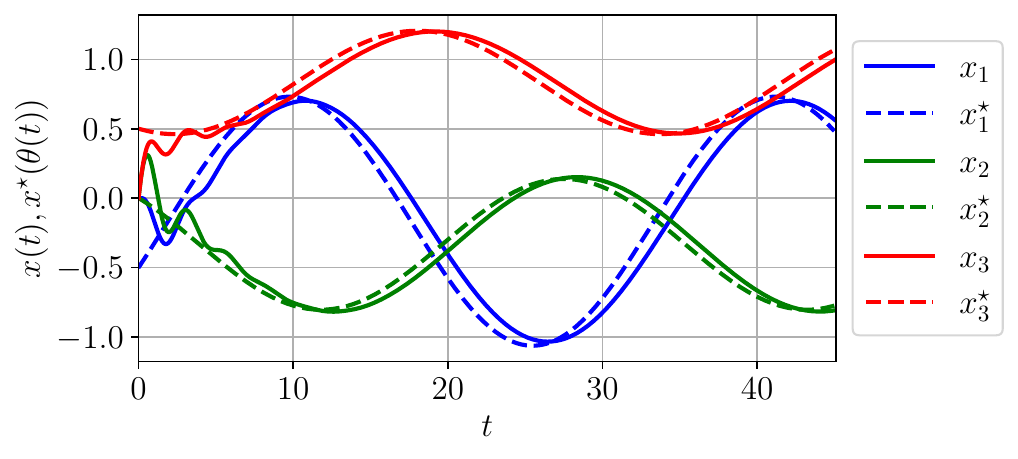}& \includegraphics[width=0.48\linewidth]{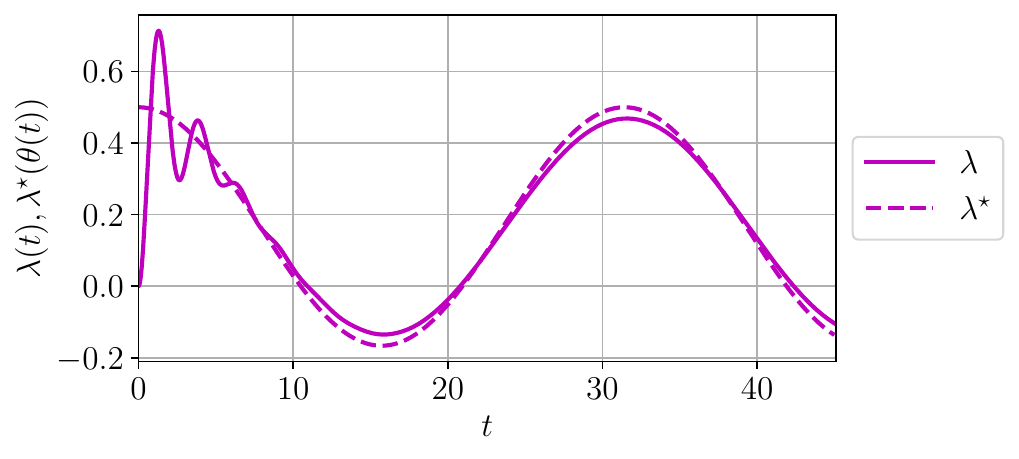}
	\end{tabular}
	\caption{Plots of trajectories of the dynamics~\eqref{eq:dyn-num-sim-1} solving the equality-constrained minimization problem~\eqref{eq:num-sim-1}. The left figure shows the trajectories of the 
    primal variables $x(t)$ as solid curves and the trajectories of the instantaneously optimal primal variables $\xstar(\theta(t))$ as dashed curves. The right figure shows the trajectory of the dual variable $\lambda(t)$ as a solid curve and the instantaneously optimal dual variable $\lambda^\star(\theta(t))$ as a dashed curve.}\label{fig:eq-constrained-sim}
\end{figure*}

\begin{figure}[th!]
    \centering
    \includegraphics[width=0.9\linewidth]{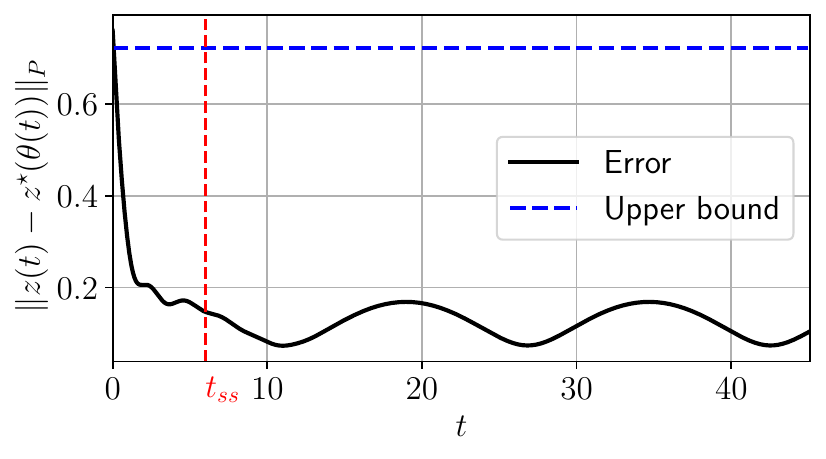}
    \caption{
    Plot of the error $\|z(t) - z^{\star}(\theta(t))\|_{P}$ along with the upper bound $\omega\ell/c^2$ where $z(t) = (x(t),\lambda(t))$ and $z^\star(\theta(t)) = (\xstar(\theta(t)),\lambda^\star(\theta(t)))$. We also denote $3$ time-constants by $\subscr{t}{ss}$, where one time-constant is $1/c = 2$ units of time.
    }\label{fig:eq-constrained-bound}
\end{figure}

\paragraph*{\change{Parameter-varying case}}
{Consider the \emph{parameter-dependent composite minimization problem}
\begin{equation}\label{eq:theta-composite}
    \min_{x \in \real^n} f_\theta(x) + g_\theta(Ax), 
    \end{equation}
where for each $\theta \in \Theta$, the function $f_\theta$ is continuously differentiable, strongly convex, and strongly smooth, the map $g_\theta$ is CCP, \change{and} that $A$ is full row rank. Then the \change{proximal augmented Lagrangian primal-dual dynamics} are
\begin{equation}\label{eq:theta-pal}
\begin{aligned}
\dot{x} &= -\nabla f_\theta(x) - A^\top \nabla M_{\gamma g_{\theta}}(Ax + \gamma \lambda), \\
\dot{\lambda} &= \gamma(-\lambda + \nabla M_{\gamma g_{\theta}}(Ax + \gamma \lambda)).
\end{aligned}
\end{equation}
For each $\theta \in \Theta$, the minimization problem~\eqref{eq:theta-composite} has a unique minimizer $\xstar(\theta)$ and Lagrange multiplier $\lambda^\star(\theta)$ and the dynamics~\eqref{eq:theta-pal} converge to them. Let $\map{\fpal}{\real^{n+m} \times \Theta}{\real^{n+m}}$ denote the vector field for the dynamics~\eqref{eq:theta-pal}. When $\map{\theta}{\realnonnegative}{\Theta}$ is a differentiable, time-varying parameter, under the assumption that $\theta \mapsto \fpal(z,\theta)$ is Lipschitz uniformly in $z$, the dynamics~\eqref{eq:theta-pal} are guaranteed to track $\xstar(\theta(t)), \lambda^\star(\theta(t))$ with a tracking error proportional to $\|\dot{\theta}(t)\|$ after a transient. The assumption that $\fpal$ is Lipschitz in $\theta$ is satisfied if, e.g., $\theta \mapsto \nabla f_\theta(x)$ and $\theta \mapsto \prox{\gamma g_\theta}(x)$ are Lipschitz uniformly in $x$. Finally, if $\fpal$ is differentiable in both of its arguments, then we can leverage Theorem~\ref{thm:exact=tracking} to design a feedforward term involving $\dot{\theta}$ to attain zero tracking error.
}

\oldchange{As in the case of linear equality constrained minimization, we have not let the matrix $A$ depend on the parameter $\theta$. For the same reason as before, the norm with respect to which the dynamics~\eqref{eq:prox-aug-primal-dual} are contracting depends on $A$.}

\section{Numerical and Hardware Experiments}
\label{sec:sim}
In this section, we \change{present experiments} to
showcase the performance of the proposed dynamics.\footnote{\change{Code for our experiments is available at \url{https://github.com/davydovalexander/time-varying-convex}.}} We present \oldchange{an application of Theorem~\ref{thm:bound_time_varying_eq} to a problem} with equality constraints, which corresponds to Problem~\ref{prob:eq-constrained} and a case with inequality constraints, which corresponds to Problem~\ref{prob:general-two-objective}.
Additionally, we consider an application to control barrier function-based design,~\cite{ADA-XX-JWG-PT:17}, for collision avoidance in a multi-robot scenario \change{by leveraging the contraction analysis of the proximal gradient dynamics} from Problem~\ref{prob:splitting}.

\begin{figure*}
	\centering
	\begin{tabular}{cc} \includegraphics[width=0.48\linewidth]{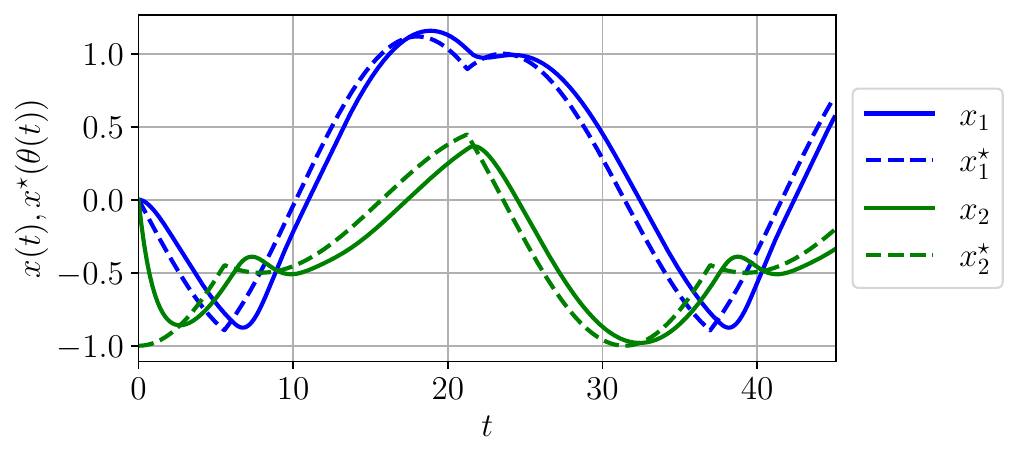}& \includegraphics[width=0.48\linewidth]{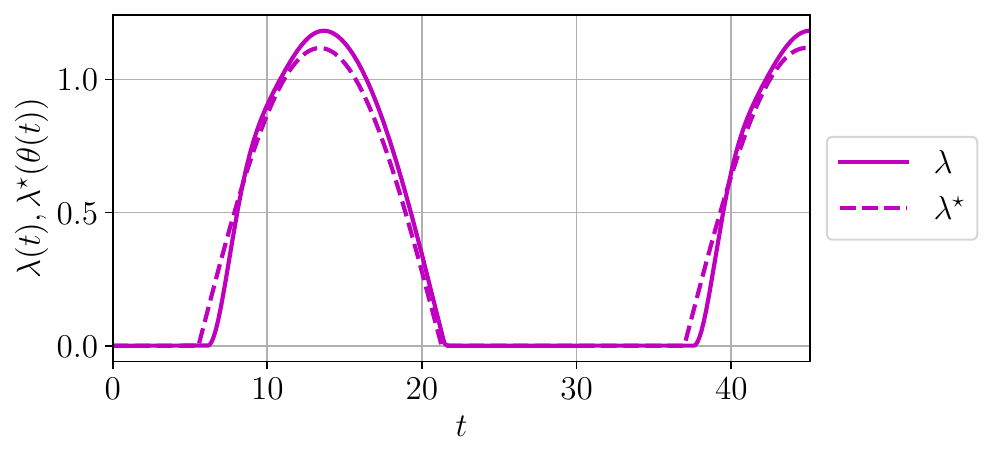}
	\end{tabular}
	\caption{Plots of trajectories of the dynamics~\eqref{eq:dyn-num-sim-2} solving the inequality-constrained minimization problem~\eqref{eq:num-sim-2}. The left figure shows the trajectories of the $2$ primal variables $x(t)$ as solid curves and the trajectories of the instantaneously optimal primal variables $\xstar(\theta(t))$ as dashed curves. The right figure shows the trajectory of the dual variable $\lambda(t)$ as a solid curve and the instantaneously optimal dual variable $\lambda^\star(\theta(t))$ as a dashed curve.}\label{fig:ineq-constrained-sim}
\end{figure*}
\begin{figure}[t!] \centering\includegraphics[width=0.9\linewidth]{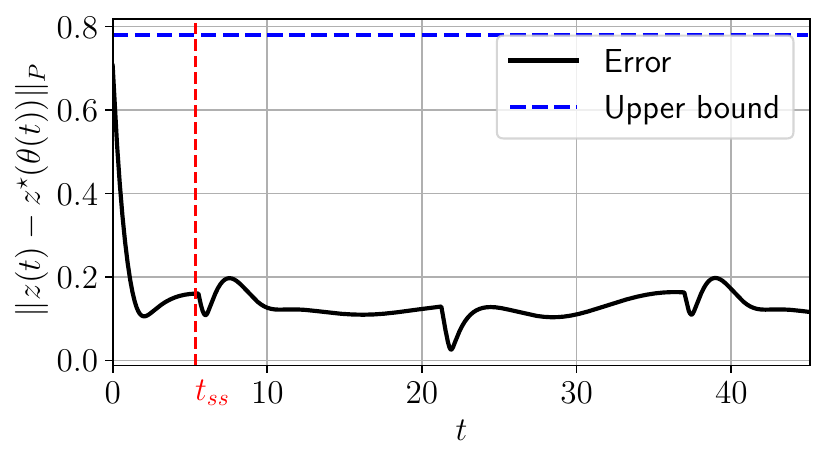}
    \caption{
    Plot of the error $\|z(t) - z^{\star}(\theta(t))\|_{P}$ along with the upper bound $\omega\ell/c^2$ where $z(t) = (x(t),\lambda(t))$ and $z^\star(\theta(t)) = (\xstar(\theta(t)),\lambda^\star(\theta(t)))$. We also denote $3$ time-constants by $\subscr{t}{ss}$, where one time-constant is $1/c\approx 1.78$ units of time.
    }\label{fig:ineq-constrained-bound}
\end{figure}

\subsection{Numerical Experiment: Equality Constraints}\label{sec:equality-ex}
Consider the following time-varying quadratic optimization problem with equality constraints
\begin{equation}\label{eq:num-sim-1}
    \begin{aligned}
\min_{x\in\real^3}\quad & \frac{1}{2}\|x - r(t)\|_2^2, \\
\text{s.t. } \quad & x_1 + 2x_2 + x_3 = \sin(\omega t),
\end{aligned}
\end{equation}
where $\omega = 0.2$ and $r(t) = (\sin(\omega t), \cos(\omega t), 1)$. We can see that~\eqref{eq:num-sim-1} is an instance of~\eqref{eq:theta-equality-constrained} with $n=3$ primal variables and $m = 1$ equality constraints by letting $\theta(t) = (\cos(\omega t), \sin(\omega t)) \in \Theta :=\setdef{z \in \real^2}{\|z\|_2 \leq 1} \subset \real^2$. Letting $\|\cdot\|_\Theta = \|\cdot\|_2$, 
\oldchange{the Lipschitz assumptions are} verified for $f_{\theta(t)}(x) = \frac{1}{2}\|x - r(t)\|_2^2$ and $b_{\theta(t)} = \sin(\omega t)$. 
The corresponding primal-dual dynamics for 
problem~\eqref{eq:num-sim-1} read
\begin{equation}\label{eq:dyn-num-sim-1}
\begin{aligned}
    \dot{x}_1 &= -x_1 + \sin(\omega t) - \lambda, \\ 
    \dot{x}_2 &= -x_2 + \cos(\omega t) - 2\lambda, \\
    \dot{x}_3 &= -x_3 + 1 - \lambda, \\
    \dot{\lambda} &= x_1 + 2x_2 + x_3 - \sin(\omega t),
\end{aligned}
\end{equation}
i.e., the dynamics are that of a linear system. 

We simulate the dynamics~\eqref{eq:dyn-num-sim-1} over the time interval $t \in [0, 45]$ with a forward Euler discretization with stepsize $\Delta t = 0.01$ and set the initial conditions $x(0) = \vectorzeros[3], \lambda(0) = 0$. We plot the trajectories of the dynamics along with the instantaneously optimal values $\xstar(\theta(t)), \lambda^\star(\theta(t))$ in Figure~\ref{fig:eq-constrained-sim}. 

We empirically observe how the trajectories of the dynamics track the
instantaneously optimal values $\xstar(\theta(t)),
\lambda^\star({\theta(t)})$ after a small transient. We then verify that
the bound from Theorem~\ref{thm:bound_time_varying_eq} provides valid upper
bounds for the tracking error.  Finding the norm with respect to which the
stable linear system~\eqref{eq:dyn-num-sim-1} is contracting with largest
rate corresponds to a bisection algorithm and is detailed
in~\cite[Section~2.5.2]{FB:24-CTDS}. After executing the bisection algorithm,
we find that the primal-dual dynamics for~\eqref{eq:num-sim-1} are strongly
infinitesimally contracting with respect to $\|\cdot\|_{P}$ with rate $c =
0.5$ for suitably chosen $P = P^\top \succ 0$. Then the corresponding
Lipschitz constant for the vector field is computed from
$(\Theta,\|\cdot\|_2)$ to $(\real^4, \|\cdot\|_{P})$, and is approximately
$\lf \approx 0.902$. From Theorem~\ref{thm:bound_time_varying_eq}, we know
that the asymptotic tracking error as measured in the $\|\cdot\|_{P}$ norm
is upper bounded by $\frac{\lf}{c^2}\omega \approx 0.722$ since
$\|\dot{\theta}(t)\|_2 = \omega$ for all $t \geq 0$. In
Figure~\ref{fig:eq-constrained-bound} we plot
$\left\|z(t)-z^{\star}(\theta(t))\right\|_{P}$, where $z$ is the stacked
vector of $x$ and $\lambda$, as well as the upper bound
$\frac{\lf}{c^2}\omega$ to demonstrate the validity of our bound.
        
\begin{remark}
    Note that in this example we have leveraged the fact that the dynamics~\eqref{eq:dyn-num-sim-1} are linear to get improved rates of contraction. If we had instead simply used the bound on the contraction rate from Theorem~\ref{thm:primal-dual-contractivity}, we would instead have $c = \oldchange{\frac{1}{4}}$, which would yield looser bounds on the asymptotic error (measured in a different norm, however).  \oprocend
\end{remark}

\subsection{Numerical Experiment: Inequality Constraints}\label{sec:inequality-ex}
Consider the following time-varying quadratic optimization problem with inequality constraints
\begin{equation}\label{eq:num-sim-2}
    \begin{aligned}
\min_{x\in\real^2}\quad & \frac{1}{2}\|x + r(t)\|_2^2, \\
\text{s.t. } \quad & {-}x_1 + x_2 \leq \cos(\omega t),
\end{aligned}
\end{equation}
where ${\omega = 0.2}$ and ${r(t) = (\sin(\omega t),\cos(\omega t))}$. We see that~\eqref{eq:num-sim-2} is an instance of~\eqref{eq:theta-composite}
with ${\map{g}{\real \times \Theta}{\realextended}}$ given by ${g(z,\theta) = \iota_{\mcC_\theta}(z)}$ where ${\mcC_\theta = \setdef{z \in \real}{z \leq \theta_1}}$, $\theta(t) = {(\cos(\omega t), \sin(\omega t)) \in  \Theta := \setdef{z \in \real^2}{\|z\|_2 \leq 1}}$, and ${A = [-1, 1]}$. 
\change{Following the expressions~\eqref{eq:prox-aug-inequality}, the correspoding \proxineqname} 
read
\begin{align}
    \dot{x}_1 &= -x_1 - \sin(\omega t) + \frac{1}{\gamma}\relu(-x_1 + x_2 + \gamma\lambda - \cos(\omega t)), \nonumber\\
    \dot{x}_2 &= -x_2 - \cos(\omega t) - \frac{1}{\gamma}\relu(-x_1 + x_2 + \gamma\lambda - \cos(\omega t)), \nonumber \\
    \dot{\lambda} &= -\gamma\lambda + \relu(-x_1 + x_2 + \gamma\lambda - \cos(\omega t)).\label{eq:dyn-num-sim-2}
\end{align}

We simulate the dynamics~\eqref{eq:dyn-num-sim-2} with $\gamma = 10$ over
the time interval $t \in [0, 45]$ with a forward Euler discretization with stepsize $\Delta t = 0.01$ and 
initial conditions ${x(0) = \vectorzeros[2], \lambda(0) = 0}$. We plot the trajectories of the dynamics along with the instantaneously optimal values $\xstar(\theta(t)), \lambda^\star(\theta(t))$ in Figure~\ref{fig:ineq-constrained-sim}. 

We empirically observe how the trajectories of the dynamics track the instantaneously optimal values $\xstar(\theta(t)), \lambda^\star({\theta(t)})$ after a small transient. We then verify that the bound from Theorem~\ref{thm:bound_time_varying_eq} provides a valid upper bound for the tracking error. Note that 
the vector field for the dynamics~\eqref{eq:dyn-num-sim-2} 
is almost everywhere differentiable and its Jacobian has the structure
\begin{equation}\label{eq:jac-num-sim2}
    \jac{\fpal}(z) = \begin{bmatrix}
        -1 - \frac{1}{\gamma}J & \frac{1}{\gamma}J & J \\ \frac{1}{\gamma} J & -1 - \frac{1}{\gamma}J & -J \\ -J & J & -\gamma + \gamma J
    \end{bmatrix},
\end{equation}
where $J$ denotes the derivative of the $\relu$ evaluated at $-x_1 + x_2 + \gamma \lambda - \cos(\omega t)$ and always takes value in $\{0,1\}$ when it is defined. In other words, the Jacobian always takes one of two values. Finding the norm which maximizes the contraction rate of the dynamics~\eqref{eq:dyn-num-sim-2} corresponds to the minimization problem
\begin{equation}\label{eq:num-sim-bisect}
    \begin{aligned}
\min_{c\in\real, P \in \real^{3 \times 3}}\quad & c, \\
\text{s.t. } \quad & PD_i + D_i^\top P \preceq 2cP, \quad i \in \{1,2\},\\
& P = P^\top \succ 0,
\end{aligned}
\end{equation}
where $D_1$ corresponds to the Jacobian~\eqref{eq:jac-num-sim2} with $J = 1$ and $D_2$ corresponds to the Jacobian~\eqref{eq:jac-num-sim2} with $J = 0$. The problem~\eqref{eq:num-sim-bisect} can be solved using a bisection algorithm on $c$ as discussed in~\cite[Section~2.5.2]{FB:24-CTDS}. After running the bisection algorithm, we find that the dynamics~\eqref{eq:dyn-num-sim-2} are strongly infinitesimally contracting with respect to $\|\cdot\|_{P}$ with rate $c = 0.5625$ for suitably chosen $P = P^\top \succ 0$. Then the corresponding Lipschitz constant for the vector field is computed from $(\Theta,\|\cdot\|_2)$ to $(\real^3, \|\cdot\|_{P})$ and is approximately $\ell \approx 1.235$. From Theorem~\ref{thm:bound_time_varying_eq}, we know that the asymptotic tracking error as measured in the $\|\cdot\|_{P}$ norm is upper bounded by $\frac{\ell}{c^2}\omega \approx 0.781$ since $\|\dot{\theta}(t)\|_2 = \omega$ for all $t \geq 0$. In Figure~\ref{fig:ineq-constrained-bound} we plot $\left\|z(t)-z^{\star}(\theta(t))\right\|_{P}$ as well as the upper bound $\frac{\lf}{c^2}\omega$ to demonstrate the validity of our bound. 

\begin{remark}
    As in the previous example, we leverage the specific structure of the dynamics~\eqref{eq:dyn-num-sim-2} to yield sharper bounds on the contraction rate $c$ instead of using the nonlinear program~\eqref{eq:nonlinear-prog}. \oprocend
\end{remark}

\subsection{Hardware Experiment: Collision Avoidance with Online Control Barrier Functions}\label{sec:CBF}

\begin{figure*}[!tbh]
	\centering
	\begin{tabular}{cc}
		\includegraphics[align=c,width=0.47\linewidth, trim={4.5cm 0 4.5cm 3.5cm},clip]{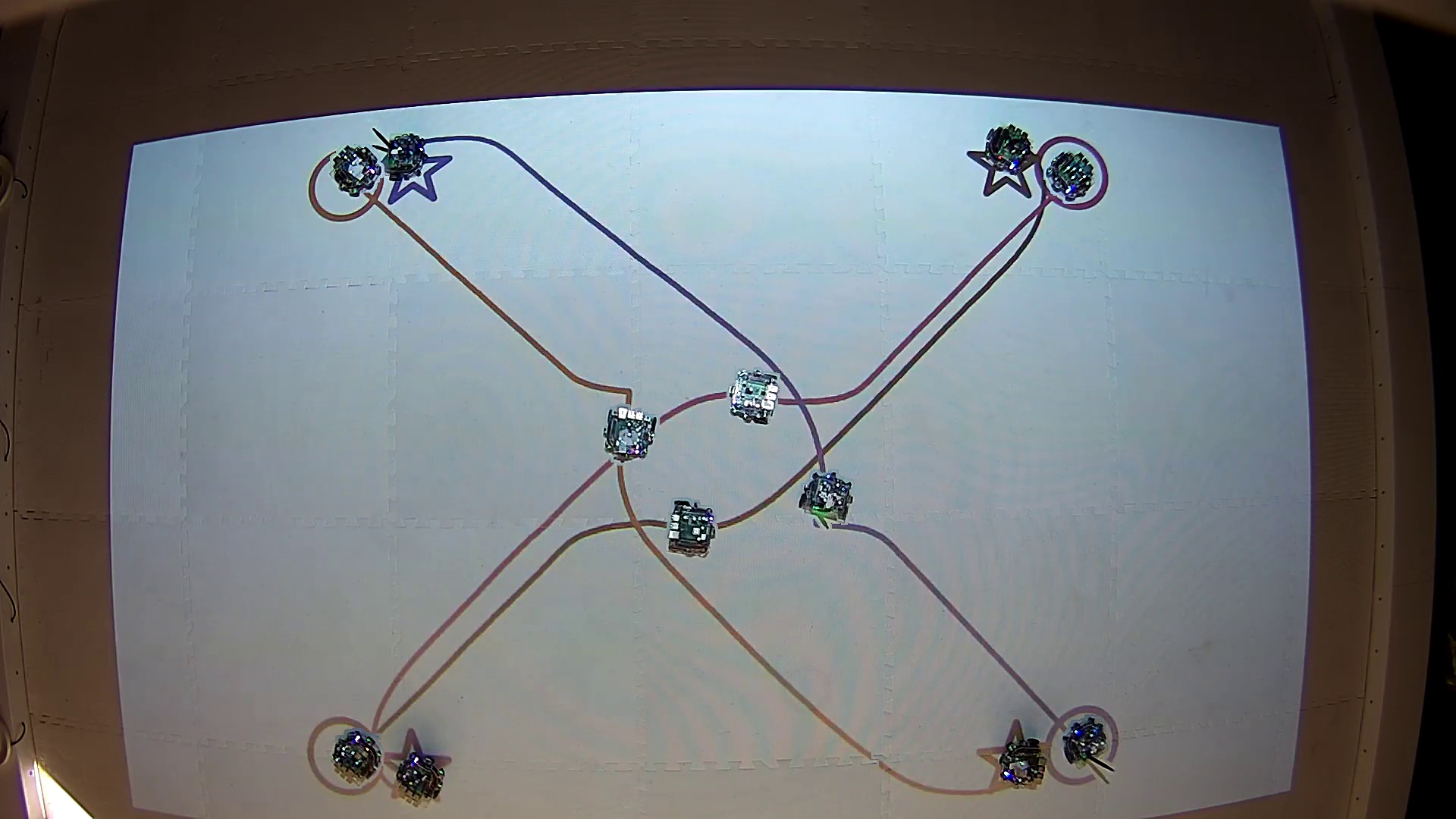} &
        \includegraphics[align=c,width=0.47\linewidth, trim={4.5cm 0 4.5cm 3.5cm},clip]{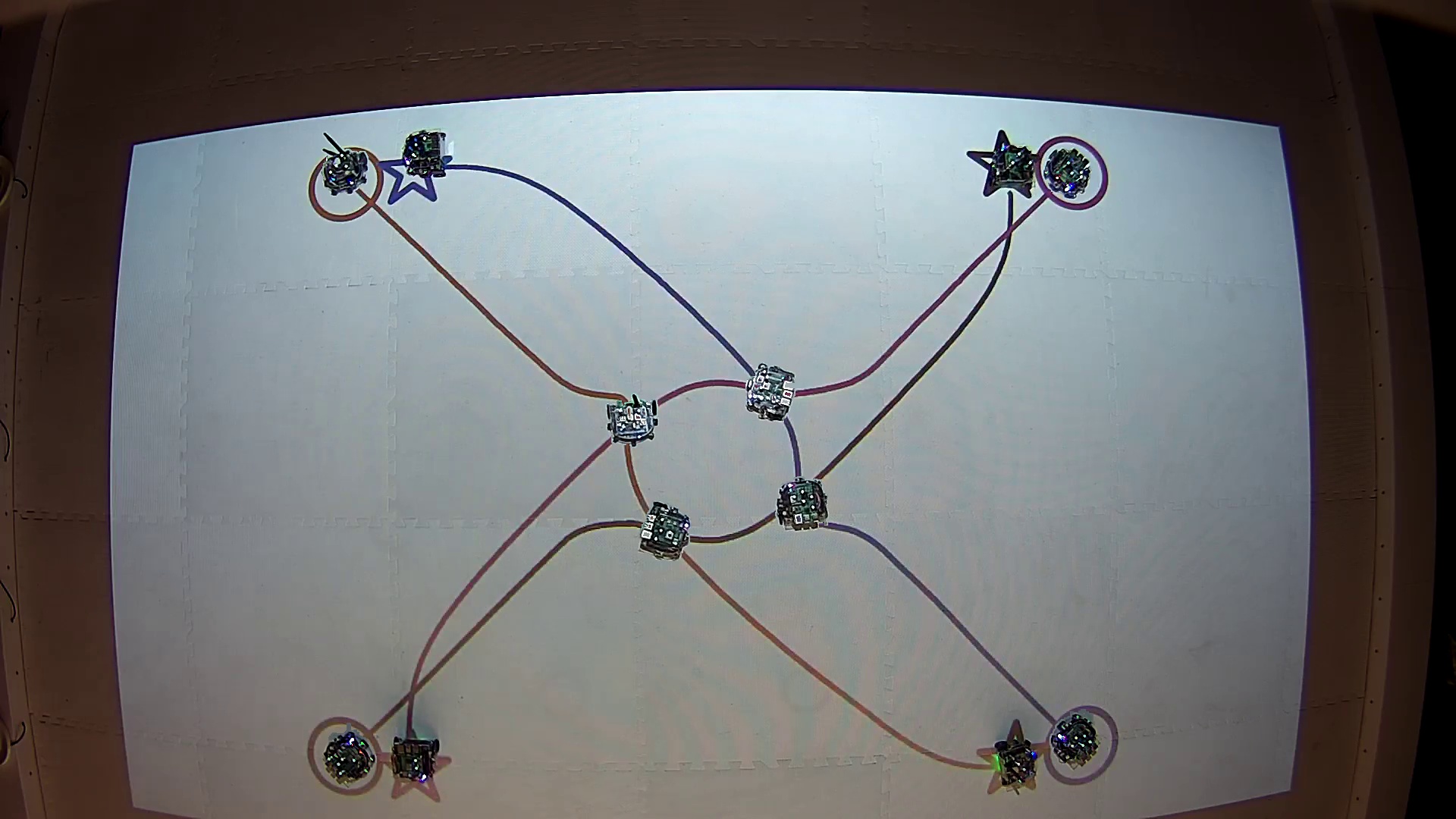} 
	\end{tabular}
 \caption{ \change{Overhead trajectories from the Robotarium experiments shown as
   photocomposites. The robots’ positions are shown at: the start of the
   maneuver (denoted by circles), the midpoint, and the end of the maneuver
   (denoted by stars). (Left) Robots executing the
   dynamics~\eqref{eq:ex3-no-feedforward} and (Right) robots executing the
   dynamics with feedforward prediction~\eqref{eq:cbf-feedforward}.  Videos
   are available at \url{https://bit.ly/TimeVaryingConvex}.}
 }\label{fig:cbf}
\end{figure*}

\begin{figure}
    \centering
    \includegraphics[width=0.9\linewidth]{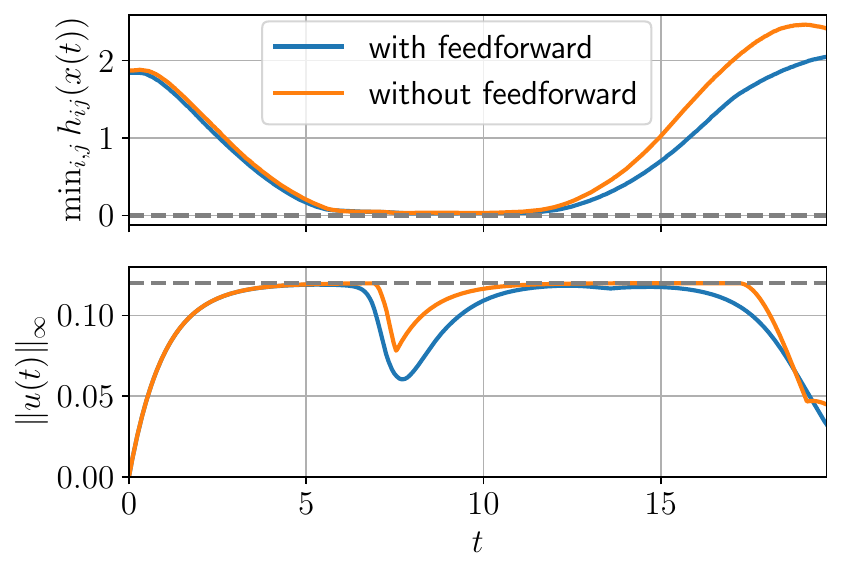}
    \caption{\change{The evolution of $\min_{i,j} h_{ij}(x(t))$ and $\|u(t)\|_{\infty}$ for the experiments in Figure~\ref{fig:cbf}. We plot $\overline{u} = 0.12$ [m/s] as a dashed line.}}\label{fig:cbf-metrics}
\end{figure}

\oldchange{We evaluate our equilibrium tracking framework in the context of control-barrier function (CBF)-based control synthesis~\cite{ADA-XX-JWG-PT:17}. To showcase our framework, we first recall the definition of a CBF.}

\oldchange{Suppose we are given a nonlinear control-affine system
\begin{equation}\label{eq:nonlinear-dyn}
    \dot{x} = F(x) + G(x)u,
\end{equation}
where $\map{F}{\real^n}{\real^n}$ and $\map{G}{\real^n}{\real^{n \times m}}$ are locally Lipschitz, {and $u \in \R^m$}. We let $\map{h}{\real^n}{\real}$ be sufficiently smooth and $\mcC := \setdef{x \in \real^n}{h(x) \geq 0}$ denote the \emph{safe set} for the system~\eqref{eq:nonlinear-dyn}.}
\begin{defn}[{Control Barrier Function~\cite[Definition~3]{ADA-XX-JWG-PT:17}}]
	\oldchange{The function $h$ is a control barrier function (CBF) for $\mcC$ if} \oldchange{there exists a locally Lipschitz and strictly increasing function $\map{\alpha}{\real}{\real}$ with $\alpha(0) = 0$ such that for all $x \in \mcC$, there exists $u \in \real^m$ with}
\begin{equation}\label{eq:CBF-condition}
		\oldchange{\nabla h(x)^\top F(x) + \nabla h(x)^\top G(x) u + \alpha(h(x)) \geq 0.}
	\end{equation}
\end{defn}
\oldchange{It is then known that a continuous feedback controller $\map{u}{\real^n}{\real^m}$ which satisfies~\eqref{eq:CBF-condition} for all $x \in \mathcal{D} \supset \mcC$, for an open set $\mathcal{D}$, renders $\mcC$ forward-invariant under the dynamics~\eqref{eq:nonlinear-dyn}~\cite{RK-ADA-SC:21}.}

\oldchange{If $\mcC = \cap_{i=1}^k \mcC_i$ and $\mcC_i := \setdef{x \in \real^n}{h_i(x) \geq 0}$, a common approach to synthesize controllers which yield $\mcC$ forward-invariant is to solve a parametric quadratic program}
\oldchange{\begin{equation}\label{eq:num-sim-3}
    \begin{aligned}
\argmin_{u\in\real^m}\quad & \frac{1}{2}\|u - \subscr{u}{nom}(x)\|_2^2, \\
\text{s.t. } \quad & a_i(x)^\top u \leq b_i(x), \quad i \in \until{k} \\
& \|u\|_\infty \leq \overline{u},
\end{aligned}
\end{equation}
where $a_i(x) = -G(x)^\top\nabla h_i(x)$ and $b_i(x) = \nabla h_i(x)^\top F(x) + \alpha(h_i(x))$ for all $i \in \until{k}$, \change{$\map{\subscr{u}{nom}}{\real^n}{\real^m}$ is a nominal feedback controller}, and $\|u\|_\infty \leq \overline{u}$, \change{where $\overline{u} > 0$,} captures actuator constraints.}

\oldchange{The complexity of the above approach arises in the computational burden of solving~\eqref{eq:num-sim-3} at each $x(t)$. To ameliorate this burden, we propose leveraging the equilibrium tracking approach, as was done in~\cite{ZM-FB-AGA:23r}. First, we relax the $k$ inequality constraints with log-barrier penalties:
\begin{align}
\argmin_{u\in\real^m}\;\; & \frac{1}{2}\|u - \subscr{u}{nom}(x)\|_2^2 \change{\;-\;} \eta(t)\sum_{i=1}^k \log(b_i(x) - a_i(x)^\top u), \nonumber \\
\text{s.t. } \;\; & \|u\|_\infty \leq \overline{u},\label{eq:log-barrier}
\end{align}
where $\map{\eta}{\real}{\realnonnegative}$ is a smooth function with $\lim_{t \to \infty} \eta(t) = 0$ which relaxes the log barriers over time. Note that we do not add additional log barrier terms for the actuator constraint since doing so would add $2m$ additional terms to the objective and would result in overly conservative controls initially. }

\oldchange{Let $\map{f}{\real^m \times \real^n \times \real}{\real}$ denote the objective function of~\eqref{eq:log-barrier}, that is, $f(u,x,\eta) = \frac{1}{2}\|u - \subscr{u}{nom}(x)\|_2^2 \change{\;-\;} \eta\sum_{i=1}^k \log(b_i(x) - a_i(x)^\top u)$. Note that both $x$ and $\eta$ are time-varying parameters in this optimization problem. For $\mcU = \setdef{u \in \real^m}{\|u\|_{\infty} \leq \overline{u}}$, we leverage the proximal gradient dynamics~\eqref{eq:prox-grad-dynamics} to track the solution $u^\star(x(t),\eta(t))$ of~\eqref{eq:log-barrier}:
\begin{equation}\label{eq:ex3-no-feedforward}
    \dot{u}(t) = -u(t) + \proj_{\mcU}\bigl(u(t) - \gamma \nabla_u f\bigr(u(t),x(t),\eta(t) \bigr) \bigr).
\end{equation}}
\change{To ensure contraction of the dynamics~\eqref{eq:ex3-no-feedforward}, we need to make two technical assumptions. First, we assume that there exist $\delta, \gamma_c > 0$ such that for all $\gamma \leq \gamma_c$, the set $\setdef{(u,x) \in \real^{n} \times \real^m}{a_i(x)^\top u \leq b_i(x) - \delta, \text{ for all } i}$ is forward-invariant for the coupled dynamics~\eqref{eq:nonlinear-dyn} and~\eqref{eq:ex3-no-feedforward}.\footnote{\change{This first assumption is reminiscient of a strengthened version of the CBF condition~\eqref{eq:CBF-condition} as studied, for example, in~\cite{RK-ADA-SC:21}.}} Secondly, we assume that each $a_i$ is bounded on this set. Under these technical assumptions, on this forward-invariant set, $f(\cdot,x,\eta)$ is strongly convex and strongly smooth. Then, in light of Theorems~\ref{thm:bound_time_varying_eq} and~\ref{thm:contractivity-forwardbackward}, we can find suitable $\gamma \in {]0,\gamma_c]}$ such that the dynamics~\eqref{eq:ex3-no-feedforward} are contracting with respect to $\|\cdot\|_2$ and will track the optimal $u^\star(x(t),\eta(t))$ solving~\eqref{eq:log-barrier} with some error depending on $\dot{x}(t)$ and $\dot{\eta}(t)$.}
\begin{arxiv}
This is the approach that is taken in~\cite{ZM-FB-AGA:23r}, albeit with the primal dual flow on the proximal augmented Lagrangian instead of the proximal gradient flow.
In contrast, we will leverage Theorem~\ref{thm:exact=tracking} to minimize tracking error.
\end{arxiv}
To apply Theorem~\ref{thm:exact=tracking}, we first replace $\proj_{\mcU}$ with a smooth approximation, $\map{\Sigma}{\real^m}{\real^m}$ satisfying $\Sigma(z) \in \mcU$ and $0 \preceq \jac{\Sigma}(z) \preceq I_n$ for all $z$.
Then we consider the dynamics {with feedforward prediction}, omitting the dependencies on time 
\begin{align}
  &\dot{u} = -u + \Sigma(y) \label{eq:cbf-feedforward} \\ &{+}
  \bigl( I_{{m}} - \jac{\Sigma}(y)(I_{{m}} - \gamma \nabla_u^2 f(u,\theta))\bigr)^{-1}\gamma\jac{\Sigma}(y)\frac{\partial \nabla_u f}{\partial \theta}(u,\theta) \dot{\theta}, \nonumber
\end{align}
where we have used the shorthands $y = u - \gamma \nabla_u f(u,\theta)$ and $\theta = (x, \eta)$. 
\begin{tac}
\change{The incremental ISS property of contracting dynamics implies that all trajectories of~\eqref{eq:cbf-feedforward} converge exponentially quickly to a small tube centered at the time-varying equilibrium trajectory of~\eqref{eq:ex3-no-feedforward}.}\footnote{\change{We refer to our technical report for more details: \url{https://arxiv.org/abs/2305.15595}.}} 
\end{tac}
\begin{arxiv}
    We can show that all trajectories of~\eqref{eq:cbf-feedforward} converge exponentially quickly to a small tube centered at the time-varying equilibrium trajectory of~\eqref{eq:ex3-no-feedforward}. To see this rigorously, suppose $\|\Sigma(z) - \proj_{\mcU}(z)\|_2 \leq \varepsilon$ for all $z$. Intuitively, for any $\varepsilon > 0$, such a $\Sigma$ exists by ``smoothing" $\proj_{\mcU}$ at the points of nonsmoothness. Then, letting $\theta = (x,\eta)$, at fixed $\theta$, we can consider the following two systems:

\begin{align}
    \dot{u}_1 &= -u_1 + \proj_{\mcU}(u_1 - \gamma \nabla_u f(u_1,\theta))\label{eq:nominal} \\
    \dot{u}_2 &= -u_2 + \proj_{\mcU}(u_2 - \gamma \nabla_u f(u_2,\theta)) + d(u_2,\theta),\label{eq:smoothed}
\end{align}
where $d(u_2,\theta) = \Sigma(y) - \proj_\mcU(y)$ and $y = u_2 - \gamma \nabla_u f(u_2,\theta)$. In other words, the smoothed dynamics~\eqref{eq:smoothed} are a perturbed version of the nominal, nonsmooth, dynamics~\eqref{eq:nominal}. We know that both of these dynamics are contracting in view of Theorem 2 and, therefore, admit unique equilibria at each $\theta$. Let $u_1^\star(\theta), u_2^\star(\theta)$ denote these parametrized equilibria. From~\cite[Theorem~38(ii)]{AD-SJ-FB:20o}, if $c - \varepsilon > 0$, we know that the following bound holds:
\begin{equation}
    \|u_1^\star(\theta) - u_2^\star(\theta)\|_2 \leq \frac{\|d(u_2^\star,\theta)\|_2}{c - \varepsilon} \leq \frac{\varepsilon}{c - \varepsilon}.
\end{equation}
In other words, as $\varepsilon \to 0$, then $\|u_1^\star(\theta) - u_2^\star(\theta)\|_2 \to 0$. 

Extending to the case of time-varying $\theta$, we have that, for all $t$, $\|u_1^\star(\theta(t)) - u_2^\star(\theta(t))\|_2 \leq \frac{\varepsilon}{c - \varepsilon}.$
Namely, the time-varying equilibrium trajectory of \eqref{eq:smoothed} lies in a small tube around the time-varying equilibrium trajectory of~\eqref{eq:nominal}. Then since~\eqref{eq:smoothed} has smooth dynamics, we may add a feedforward term to~\eqref{eq:smoothed} so that all trajectories converge to $u_2^\star(\theta(t))$ exponentially fast. And since $u_2^\star(\theta(t))$ lies inside an arbitrarily small tube around $u_1^\star(\theta(t))$, we can achieve arbitrarily close tracking depending on how we tune $\varepsilon$.
\end{arxiv}
{Recall that $\dot{x}$ follows the dynamics~\eqref{eq:nonlinear-dyn} and that since $\eta$ is a design choice, $\dot{\eta}$ is known as well, which makes these dynamics well-posed.}

\oldchange{We now 
{provide} a concrete example. 
Consider a team of $n$ single-integrator robots, $\dot{x}_i = u_i$, in $\real^2$ attempting to avoid collisions with one another while driving from an initial position to a final one.
\change{Let $x = (x_1,\dots,x_n) \in \real^{2n}$ and $u = (u_1,\dots,u_n) \in \real^{2n}$ denote the concatenation of their states and controls, respectively.}
There are $n(n-1)/2$ CBFs, given by $h_{ij}(x) = \|x_i - x_j\|_2^2 - d^2,$ for all $i,j \in \until{n}, i \neq j$ and $d > 0$ is some safety distance between agents. The constraints induced by these CBFs are $-\nabla h_{ij}(x)^\top u \leq \alpha(h_{ij}(x))$.}

\change{We implement our proposed controllers in the Robotarium~\cite{SW-PG-LW-SM-GN-MM-ME:20}, a remotely accessible multi-robot testbed at the Georgia Institute of Technology. We consider $n = 4$ robots and evolve the dynamics~\eqref{eq:ex3-no-feedforward} and~\eqref{eq:cbf-feedforward} as the controllers for the robots. We choose $\eta(t) = \e^{-0.3t}, \gamma = 0.5, \alpha(r) = 3r, d = 0.182$ [m], and as nominal control, we take ${\subscr{u}{nom}}_i(x_i) = \subscr{x}{des,i} - x_i$, where $\subscr{x}{des,i} \in \real^2$ denotes the desired final position for robot $i$. To comply with actuation limits in the Robotarium, we take $\overline{u} = 0.12$ [m/s]. For the dynamics~\eqref{eq:cbf-feedforward}, we take $\Sigma(z) = \overline{u}\tanh(z/\overline{u})$ as a smooth approximation for $\proj_{\mcU}$.}
\change{The results of the experiment are shown in Figure~\ref{fig:cbf} and the evolution of $\min_{i,j} h_{ij}(x)$ and $\|u\|_{\infty}$ are shown in Figure~\ref{fig:cbf-metrics}.}

We can see that both control strategies result in safe execution where each $h_{ij}$ is nonnegative for all~$t$ and yield controls that obey the actuation constraint $\|u(t)\|_{\infty} \leq \overline{u}$ for all~$t$. Note that without the feedforward term, the robots overshoot their target destination before eventually converging to the final location while this does not happen when there is a feedforward correction, \change{see Figure~\ref{fig:cbf-metrics}}. This overshoot is due to the lack of knowledge that the log barrier penalty coefficient $\eta(t)$ is decaying to zero. In contrast, the feedforward term allows the robots to account for its exponential decay to zero and thus they avoid overshooting their goal.

\section{Discussion}
In this article, we take a contraction theory approach to the problem of tracking optimal trajectories in time-varying convex optimization problems. We prove in Theorem~\ref{thm:bound_time_varying_eq} that the tracking error between any solution trajectory of a strongly infinitesimally contracting system and its equilibrium trajectory is upper bounded with an explicit estimate on the bound. \oldchange{We additionally prove in Theorem~\ref{thm:exact=tracking} that any strongly infinitesimally contracting system can be augmented with a feedforward term to ensure that the tracking error converges to zero exponentially quickly.} To apply \oldchange{these theorems}, we establish the strong infinitesimal contractivity of three dynamical systems solving optimization problems and apply Theorem~\ref{thm:bound_time_varying_eq} to provide explicit tracking error bounds. We validate these bounds in two numerical examples \change{and present a novel application to CBF-based control}.

We believe that this work motivates future research in establishing the strong infinitesimal contractivity of dynamical systems solving optimization problems or performing more general computation due to the desirable consequences of contractivity. As future research, we plan to investigate (i) discretization of parameter-varying contracting dynamics and establish similar tracking error bounds for discrete-time contracting systems~\cite{GR-FW:22}, (ii) contractivity properties of continuous-time stochastic optimization algorithms based on stochastic differential equations~\cite{ZA:22}, and (iii) \change{time-varying} nonconvex optimization problems with isolated local minima~\change{\cite{YD-JL-MA:23}, possibly} using the theory of $k$-contraction~\cite{CW-IK-MM:22}. \changeafter{Finally, we believe that a comprehensive comparison to methods based upon incremental quadratic constraints~\cite{LL-BR-AP:16} and dissipative systems theory~\cite{LL:22} could provide novel design insights.} 

\appendices
\section{Proofs and Additional Results}
\label{app:1:Proofs and Additional Results}
\change{We begin with a result on parametrized contractions.}
\begin{lemma}[Parametrized contractions]\label{lemma:parametrized-contractions}
  \oldchange{Consider the system~\eqref{eq:system_1} satisfying Assumptions~\ref{ass:1_} and~\ref{ass:2_}.} 
  Let $\map{\xstar}{\newparamset}{\mcX}$ denote the map given by $\xstar(\newparam) =
  x_{\newparam}^\star$. Then $\xstar(\cdot)$ is
  Lipschitz from $(\newparamset,\|\cdot\|_{\newparamset})$ to $(\mcX,\|\cdot\|_{\mcX})$ with constant $\ell_{\newparam}/c$.
\end{lemma}
\begin{proof}
  Consider the dynamics $\dot{x}=F(x,\change{\theta})$ {satisfying
    Assumptions~\ref{ass:1_} and~\ref{ass:2_}}, where
  $\change{\theta}$ is constant.  Given two constant inputs
  $\change{\theta_1}$ and $\change{\theta_2}$, the two
  equilibrium solutions are $\xstar(\change{\theta_1})$ and
  $\xstar(\change{\theta_2})$. The assumptions of 
  \cite[Theorem~3.16]{FB:24-CTDS} are satisfied with $c={-}\osL_x(F)$ and
  $\ell_{\change{\theta}}=\Lip_{\change{\theta}}(F)$, and the
  differential inequality~\cite[Equation~3.39]{FB:24-CTDS} implies
  \begin{equation*}
    0 \leq -c\norm{\xstar(\change{\theta_1})-\xstar(\change{\theta_2})}{\mcX}
    + \ell_{\change{\theta}} \norm{\change{\theta_1}-\change{\theta_2}}{\change{\Theta}}.
  \end{equation*}
  This concludes the proof.
\end{proof}

\change{Next, we study}
the Lipschitzness of parametrized time-varying equilibrium trajectories and \change{obtain} a bound on their time derivatives.
\begin{lemma}[Lipschitzness of parametrized curves]
	\label{lem:bound_xstar_dot}
	Consider $\mcX \subseteq \R^n$ and $\change{\Theta} \subseteq \R^d$ with associated norms $\map{\|\cdot\|_{\mcX}}{\R^n}{\R_{\geq 0}}$ and $\map{\|\cdot\|_{\change{\Theta}}}{\R^d}{\R_{\geq 0}}$, respectively.
	Let ${\map{G}{\change{\Theta}}{\mathcal{X}}}$ be Lipschitz from $(\change{\Theta},\|\cdot\|_{\change{\Theta}})$ to $(\mcX, \|\cdot\|_{\mcX})$ with constant $\Lip(G) \geq 0$. Then for every $a,b \in \real$ with $a < b$ and every continuously differentiable $\map{\change{\theta}}{{]a,b[}}{\change{\Theta}}$,
	\begin{enumerate}
		\item \label{statement_1}
		the curve $\map{x}{{]a,b[}}{\mcX}$ given by $x(t) = G(\change{\theta}(t))$ is locally Lipschitz;
		\item \label{statement_2}
		$\|\dot{x}(t)\|_{\mcX} \leq \Lip(G) \|\change{\dot{\theta}}(t)\|_{\change{\Theta}}$, for a.e. $t\in {]a,b[}$.
	\end{enumerate}
\end{lemma}
\begin{proof}
	Item~\ref{statement_1} is a consequence of the fact that continuously differentiable mappings are locally Lipschitz and that a composition of Lipschitz mappings is Lipschitz.
	
	To prove item~\ref{statement_2} we first note that item~\ref{statement_1} implies that $\dot x(t)$ exists almost everywhere by Rademacher's theorem. Next, for all $t\in {]a,b[}$ for which $\dot{x}(t)$ exists, {we have}
	\begin{align*}
	\|\dot{x}(t)\|_\mathcal{X} &{:=} \Big\|\lim_{h \to 0^+}\!\!\! \frac{x(t+h){-}x(t)}{h}\Big\|_{\mcX}
    \!{=}\! \lim_{h \to 0^+}\!\! \frac{1}{h}\|x(t+h) {-}x(t)\|_{\mcX}\\ 
	&{\leq} \lim_{h \to 0^+}\frac{\Lip(G)}{h}\|\change{\theta}(t+h) {-}\change{\theta}(t)\|_{\change{\Theta}} \\
    &{=} \Lip(G)\Big\|\lim_{h \to 0^+}\!\!\!{\frac{\change{\theta}(t+h) {-}\change{\theta}(t)}{h}}\Big\|_{\change{\Theta}} {=} \Lip(G) \|\change{\dot \theta}(t)\|_{\change{\Theta}},
	\end{align*}
	where we have used the continuity of the norms $\|\cdot\|_{\mcX}$ and $\|\cdot\|_{\change{\Theta}}$ and the Lipschitzness of the map $G$.
\end{proof}

\begin{lemma}\label{lemma:discrete-to-continuous-contraction}
	Let $\map{\OT}{\real^n}{\real^n}$ be Lipschitz with respect to a norm $\|\cdot\|$ with constant $\Lip(\OT)$. Then the vector field $\map{F}{\real^n}{\real^n}$ defined by the dynamics
	\begin{equation}\label{eq:aux-dynamics}
	\dot{x} = -x + \OT(x) =: F(x)
	\end{equation}
	satisfies $\osL(F) \leq -1 + \Lip(\OT)$. Moreover, if $\Lip(\OT) < 1$, then $\OT$ has a unique fixed point, $x^*$, which is the unique equilibrium point of the contracting dynamics~\eqref{eq:aux-dynamics}.
\end{lemma}
\begin{proof}
  We have $\osL(F) := \osL(-\Id + \OT) = -1 + \osL(\OT) \leq -1 +
  \Lip(\OT)$, where we used the translation property of $\osL$ and the
  upper bound $\osL(\OT) \leq
  \Lip(\OT)$~\cite[Section~3.2.3]{FB:24-CTDS}. If $\Lip(\OT) < 1$, the
  Banach contraction theorem implies the existence of a unique fixed point
  and equilibrium points of~\eqref{eq:aux-dynamics} are exactly fixed
  points of $\OT$. Moreover, the dynamics~\eqref{eq:aux-dynamics} are
  contracting since $\osL(F) < 0$.
\end{proof}

\begin{lemma}[Symmetry and bounds on Jacobians]\label{lemma:symmetry-prox-Moreau}
    Let the map ${\map{g}{\real^n}{\realextended}}$ be CCP. Then for every $\gamma > 0$, $\jac{\prox{\gamma g}}(x)$ and $\Hess M_{\gamma g}(x)$ exist for a.e. $x \in \real^n$, are symmetric, and satisfy
    \begin{equation}\label{eq:bounds-Jacobians}
        0 \preceq \jac{\prox{\gamma g}}(x) \preceq I_n, \quad 0 \preceq \Hess M_{\gamma g}(x) \preceq \frac{1}{\gamma} I_n.
    \end{equation}
\end{lemma}
\begin{proof}
    First note that $\jac{\prox{\gamma g}}(x)$ and $\Hess M_{\gamma g}(x)$ exist for a.e. $x \in \real^n$ by Rademacher's theorem since $\prox{\gamma g}$ and $\nabla M_{\gamma g}$ are both Lipschitz. Furthermore, $\Hess M_{\gamma g}(x)$ is symmetric for a.e. $x$ by symmetry of second derivatives. Analogously, $\jac{\prox{\gamma g}}(x) = \change{I_n - \gamma} \Hess M_{\gamma g}(x)$ by~\eqref{eq:Moreau-gradient} so we conclude symmetry of $\jac{\prox{\gamma g}}(x)$ as well. The bounds~\eqref{eq:bounds-Jacobians} are a consequence of the fact that $\prox{\gamma g}$ and $\Id - \prox{\gamma g}$ are both firmly nonexpansive~\cite[Proposition~12.28]{HHB-PLC:17}.
\end{proof}

\newcommand{\aeps}{\alpha_{\varepsilon}}
\newcommand{\ceps}{c_{\varepsilon}}
\newcommand{\Peps}{P_{\varepsilon}}
\section{Logarithmic norm of Hurwitz saddle  matrices}
\label{app:2:Logarithmic norm of Hurwitz saddle  matrices}

In this section, we provide bounds on log norms for the saddle matrices arising from the Jacobians of the dynamics~\eqref{eq:primal-dual} and~\eqref{eq:prox-aug-primal-dual}. 

\begin{lemma}[Logarithmic norm of Hurwitz saddle  matrices]\label{lemma:saddle-matrices}
 Given $B=B^\top\in\real^{n \times n}$ and $A\in\real^{m \times n}$, with
 $m\leq n$, we consider the \emph{saddle matrix}
 \begin{equation}
   \mcB = \begin{bmatrix}
     -B & -A^\top \\
     A & 0
   \end{bmatrix} \in \real^{(m+n)\times (m+n)}.
 \end{equation}  
 Then, for each matrix pair $(B,A)$ satisfying $\bmin I_n\preceq B
  \preceq \bmax I_n$ and $\amin I_m \preceq A A^\top \preceq \amax
  I_m$, for $\bmin, \bmax, \amin,\amax \in \realpositive$, the following contractivity LMI holds:
  \begin{equation}
    \mcB^\top P + P\mcB\preceq -2 c P \quad\iff\quad \lognorm{\mcB}{P}
    \leq - c,
  \end{equation}
  where
  \begin{align}\label{def:P-alpha+c}
    P &= \begin{bmatrix}
      I_n & \alpha A^\top \\ \alpha A & 
      I_m \end{bmatrix} \succ 0, \ 
    \alpha=\frac{1}{2}\min\Big\{\frac{1}{\bmax},\frac{\bmin}{\amax}\Big\}, \
    \text{and} \\
    c&=\frac{1}{2}\alpha\amin = 
    \frac{1}{4}\min\Big\{\frac{\amin}{\bmax},\frac{\amin}{\amax}\bmin\Big\}.
    \end{align}
\end{lemma}
\begin{arxiv}
\begin{proof}
    We start by verifying that $P\succ0$. Using the Schur complement of the
    $(2,2)$ entry, we need to verify that
    \begin{align*}
      I_n - \alpha^2 A^\top A\succ 0 \iff  1 -\alpha^2 \amax >0
      \iff  \alpha^2 < 1/\amax.
    \end{align*}    
    The inequality $\alpha^2 < 1/\amax$ follows from the tighter inequality
    \oldchange{$(2\alpha)^2 \leq \displaystyle \tfrac{1}{\amax}$} which is proved as
    follows:
    \begin{align*}
      \min\Big\{\frac{1}{\bmax},\frac{\bmin}{\amax}\Big\}^2&\leq
      \min\Big\{\frac{1}{\bmax},\frac{\bmin}{\amax}\Big\} \cdot
      \max\Big\{\frac{1}{\bmax},\frac{\bmin}{\amax}\Big\}\\
      &= \frac{1}{\bmax} \cdot \frac{\bmin}{\amax} \leq \frac{1}{\amax}.
    \end{align*}
    Next, we aim to show that $Q := {-}\mcB^\top P - P\mcB -2 c P\succeq
    0$. After some bookkeeping, we compute
    \begin{align*}
    Q =  
    \begin{bmatrix}
       2 B {-} 2 \alpha  A^\top A    {-} 2 c I_n
       &
       \alpha B A^\top {-} 2 c \alpha A^\top
        \\
        A {+} \alpha AB {-}  A
       {-} 2 c \alpha A
        &  2 \alpha A A^\top       {-} 2 c I_m
    \end{bmatrix} .
    \end{align*}
    The (2,2) block satisfies the lower bound 
    \oldchange{
    \begin{align*}
      2 \alpha A A^\top {-} 2 c I_m &=
      2 \big(\tfrac{1}{2} \alpha A A^\top {-} c I_m\big) + \alpha A A^\top  \\
      & \!\!\succeq 2\big(\tfrac{1}{2}\alpha\amin {-} c \big)I_m   + \alpha A A^\top 
       {=}  \alpha A A^\top\succ 0.
    \end{align*}
    }
    Given this lower bound, we can factorize the resulting matrix as follows:
    \begin{multline*}
     Q = {-}\mcB^\top P - P\mcB -2 c P 
     \\
     \succeq
     \begin{bmatrix} I_n & \!\! 0 \\ 0 & \!\! A      \end{bmatrix}
     \underbrace{\begin{bmatrix}
       2 B {-} 2 (\alpha  A^\top A + c I_n)
       &
       \!\alpha B  {-} 2 c \alpha I_n
        \\
        \alpha B {-} 2 c \alpha I_n
        & \!\alpha I_n 
\end{bmatrix}}_{n\times n}
\begin{bmatrix} I_n & \!\! 0 \\ 0 & \!\! A^\top      \end{bmatrix}.
\end{multline*}
Since $\alpha I_n \succ 0$, it now suffices to show that the
Schur complement of the (2,2) block of $n\times n$ matrix is positive
semidefinite. We proceed as follows:
\oldchange{
   \begin{align*}      
     & 2  B - 2  (\alpha  A^\top A + c I_n) -
      \alpha \big(  B  - 2 c  I_n\big)^2 \succeq 0
     \\
     & \!\!\! \iff  \!\!  2B  -   \alpha B^2 + 4 \alpha c B \succeq
     2 (\alpha  A^\top A + c I_n) + 4 \alpha c^2 I_n
     \\
     &\!\!\! \impliedby 2B  - \alpha B^2 \succeq 2 (\alpha  A^\top A + c I_n)
     \ \text{and} \ 4 \alpha c B \succeq 4 \alpha c^2 I_n.
   \end{align*}
   To prove $2B - \alpha B^2 \succeq 2 (\alpha  A^\top A + c I_n)$,
   we upper bound the right hand side as follows:
   \begin{align*}
     2 (\alpha A^\top A +  c I_n) &\overset{\eqref{def:P-alpha+c}}{\preceq}
     \alpha (2\amax + \amin)I_n \\
     & \!\!\!\!\!\!\!\!\! \overset{\alpha\leq\tfrac{1}{2}\bmin/\amax}{\preceq}
     \frac{1}{2} \frac{\bmin}{\amax} (2\amax + \amin)I_n \preceq \frac{3}{2}\bmin I_n.
   \end{align*}
   }
   Next, since \oldchange{$\alpha\leq\frac{1}{2\bmax}$}, we know \oldchange{$- \alpha\bmax \geq
   -\tfrac{1}{2}$}. We then upper bound the left hand side as follows:
   \oldchange{
   \begin{align*}
     2B - \alpha B^2 &\succeq   2B -  \alpha\bmax B
     \succeq (2 - \tfrac{1}{2}) B \succeq \tfrac{3}{2}\bmin I_n.
   \end{align*}
   }
   Finally, the inequality \oldchange{$4 \alpha c B \succeq 4 \alpha c^2 I_n$} follows from noting \oldchange{$c\leq \tfrac{1}{4} \tfrac{\amin}{\amax} \bmin < \bmin$}.
\end{proof}
\end{arxiv}

\newcommand{\xmax}{x_{\max}}
\newcommand{\diag}{\mathrm{diag}}

\begin{arxiv}
The following lemma is presented in~\cite[Lemma~6]{GQ-NL:19}. We include it here for completeness.
\begin{lemma}\label{lemma:BoundQ22}
Let $X = X^\top \in \real^{m\times m}$ satisfy $0 \preceq X \preceq \xmax I_m$ for some $\xmax > 0$, $\gamma > 0$, and $\ds\alpha \leq \frac{\gamma}{\xmax}$. Then for all $d \in [0,1]^m$, the following inequality holds
\begin{equation}
\alpha(\diag(d) X+ X\diag(d)) + 2\gamma(I_m - \diag(d)) \succeq \frac{3}{2}\alpha X.
\end{equation}
\end{lemma}
\begin{proof}
    See~\cite[Lemma~6]{GQ-NL:19}.
\end{proof}
\end{arxiv}

The following lemma is a generalization of~\cite[Lemma~4]{GQ-NL:19}, where we let the matrix $G$ be dense. \change{The proof method is improved by introducing an auxiliary variable $\varkappa$ which we optimize for in~\eqref{eq:nonlinear-prog} whereas in the proof of~\cite[Lemma~4]{GQ-NL:19}, $\varkappa = 1$ is chosen.}

\begin{lemma}[Generalized saddle  matrices]\label{lemma:generalized-saddle}
 Given $A\in\real^{m \times n}$, ${B=B^\top\in\real^{n \times n}}$, and $G= \trasp{G} \in \real^{m\times m}$, with
 $m\leq n$, and $\gamma >0$, we consider the \emph{saddle matrix}
 \begin{equation}
   \mcB = 
    \begin{bmatrix}
        - B - \frac{1}{\gamma} \trasp{A}GA & - \trasp{A} G \\
        GA & -\gamma(I_m - G)
    \end{bmatrix} \in \real^{(m+n)\times (m+n)}.
 \end{equation}  
 Then, for each matrix triplet $(B,A,G)$ satisfying $\bmin I_n\preceq B \preceq \bmax I_n$, $\amin I_m \preceq A A^\top \preceq \amax I_m$, and $0 \preceq G \preceq I_m$, for $\bmin$, $\bmax$, $\amin$, $\amax \in \realpositive$, the following contractivity LMI holds:
  \begin{equation}
    \mcB^\top P + P\mcB\preceq -2 \cs P \quad\iff\quad \lognorm{\mcB}{P}
    \leq - \cs,
  \end{equation}
  where
  \begin{equation}\label{def:alpha+c}
    P = \begin{bmatrix}
      I_n & \as A^\top \\ \as A & 
      I_m \end{bmatrix} \succ 0, 
    \end{equation}
and $\cs$ and $\as$ are the optimal parameters for the problem~\eqref{eq:nonlinear-prog}.
\end{lemma}
\begin{arxiv}
\begin{proof}
    Define the matrix $Q \in \real^{(m+n)\times (m+n)}$ by
    \begin{align}
        Q = - \trasp{\mcB} P - P \mcB - 2\cs P = \begin{bmatrix}
            Q_{11} & Q_{12} \\
            Q_{12}^\top & Q_{22}
        \end{bmatrix},
    \end{align}
    where $Q_{11}\in \real^{n \times n}$, $Q_{12}\in \real^{n\times m}$, and $Q_{22}\in \real^{m\times m}$. We aim to show that $Q\succeq 0$ for $\cs$ and $\as$ optimal parameters for the problem~\eqref{eq:nonlinear-prog}. We have
    \begin{align*}
        Q_{11} &= 2B + \Big(\frac{2}{\gamma} - 2\as  \Big)\trasp{A}GA   -  2\cs I_n, \\
        Q_{12} &=  \as\gamma\trasp{A}(I_m {-} G) 
    {+} \as B \trasp{A} {+} \frac{\as}{\gamma} \trasp{A}GA\trasp{A}  {-} 2\cs \as \trasp{A}, \\
    Q_{22} &= \as(A\trasp{A} G + GA\trasp{A}) + 2\gamma (I_m -G) - 2\cs I_m.
    \end{align*}
    To show that $Q\succeq 0$ we use the Schur Complement, which requires to prove that $Q_{22} \succ 0$ and $Q_{11} - Q_{12}Q_{22}^{-1}\trasp{Q_{12}} \succeq 0$. We do this in three steps. 
    
    First, we find a lower bound for $Q_{22}$. Since $G$ is symmetric and satisfies $0 \preceq G  \preceq I_m$, there exists an orthogonal matrix $U \in \real^{m \times m}$ and $d \in [0,1]^m$ such that $G = U\diag(d)U^\top$.
Substituting this into $Q_{22}$
and multiplying 
on the left and on the right by $U$ and $\trasp{U}$, respectively, we get
\begin{align}
\label{eq:U^TQ22U}
    \trasp{U}Q_{22} U &= 
    \as(\trasp{U}A\trasp{A} U\diag(d) + \diag(d)\trasp{U}A\trasp{A}U) \nonumber\\
    & \quad + 2\gamma (I_m -\diag(d)) - 2\cs I_m,
\end{align}
where we have used the fact that $U\in \R^{m\times m}$ orthogonal, i.e., ${U\trasp{U} = \trasp{U}{U} = I_m}$.
Moreover, the orthogonality of $U$ implies that ${U^{-1} = \trasp{U}}$. Thus the eigenvalues of $\trasp{U}A\trasp{A}U$ and $A\trasp{A} $ are equal, and therefore, $\amin I_m\preceq\trasp{U}A\trasp{A}U \preceq \amax I_m$. Next, applying Lemma~\ref{lemma:BoundQ22} to~\eqref{eq:U^TQ22U}, with $X := U^\top AA^\top U$, 
and, multiplying this on the left and on the right by $U$ and $\trasp{U}$, respectively, we get the following lower bound
\begin{align}
    Q_{22}\succeq 
    \frac{3}{2} \as A\trasp{A} - 2\cs I_m,
\end{align}
where the inequality holds because $\as \leq \frac{\gamma}{\amax}$ -- constraint~\eqref{eq:alpha-constraint}.
Finally, we note that $Q_{22} \succ 0$ for $\as$ and $\cs$ optimal parameters for the problem~\eqref{eq:nonlinear-prog}

Next, 
we need to prove that $Q_{11} - Q_{12}Q_{22}^{-1}\trasp{Q_{12}} \succeq 0$ for $ \cs$ and $\as$ optimal parameters for the problem~\eqref{eq:nonlinear-prog}.
To this purpose, first note that
for every $\kappa > \frac{2}{3}\frac{\amax}{\amin}\frac{1}{\as}$,
\begin{align}
\label{eq:bound_Q22}
Q_{22} \succeq \frac{1}{\kappa} A\trasp{A}.
\end{align}
Next, we upper bound $Q_{12}Q_{22}^{-1}\trasp{Q_{12}}$. To simplify notation, define 
$R_1 = B + \frac{1}{\gamma} \trasp{A}GA  - 2\cs I_n \in \R^{n\times n}$, $R_2 = \gamma\trasp{A}(I_m - G) \in \R^{n\times m}$, 
and note that $Q_{12} = \as(R_1 \trasp{A} + R_2) \in \R^{n\times m}$. We compute
\begin{align}
    &\!\!\!\!\!\!Q_{12}Q_{22}^{-1}\trasp{Q_{12}} 
    \!\!\overset{\eqref{eq:bound_Q22}}{\preceq} \kappa Q_{12}(A\trasp{A})^{-1} Q_{12} \nonumber\\
    &\preceq (\as)^2\kappa(R_1\trasp{A}(A\trasp{A})^{-1}  A\trasp{R}_1 {+} R_1\trasp{A}(A\trasp{A})^{-1}\trasp{R}_2\nonumber\\ 
    & \quad + R_2(A\trasp{A})^{-1}  A\trasp{R}_1 +
    R_2(A\trasp{A})^{-1}  \trasp{R}_2)\nonumber\\
    &\preceq (\as)^2\kappa(R_1\trasp{R}_1 + 2\norm{R_1\trasp{A}(A\trasp{A})^{-1}\trasp{R}_2}{}I_n\nonumber\\ 
    & \quad +\norm{R_2(A\trasp{A})^{-1}  \trasp{R}_2}{}I_n), \label{eq:boundQ12Q22Q12^T} 
\end{align}
where the final inequality holds because $\trasp{A}(A\trasp{A})^{-1} A \preceq I_n$.
Note that
\begin{align*}
    R_1\trasp{R}_1 
    &\preceq \bmax^2 I_n {+} \Big(2\bmax{+}\frac{\amax}{\gamma} {+} 2\cs\Big)\Big(\frac{\amax}{\gamma} {+} 2\cs\Big) I_n {=:} h_1(\cs) I_n,
\end{align*}
where we have introduced the function ${\map{h_1}{\R_{\geq 0}}{\R_{\geq 0}}}$ defined by $h_1(\cs) = \bmax^2 + 2\bmax(\frac{\amax}{\gamma} + 2\cs) + (\frac{\amax}{\gamma} + 2\cs)^2$.
Moreover,
\begin{align*}
    \norm{R_2(A\trasp{A})^{-1}  \trasp{R}_2}{} &\leq \gamma^2\frac{\amax}{\amin}, \text{ and }\\
    2\norm{R_1\trasp{A}(A\trasp{A})^{-1}\trasp{R}_2}{} &\leq 2\gamma\frac{\amax}{\amin} \Big(\bmax {+} \frac{\amax}{\gamma} {+} 2\cs \Big)=: h_2(\cs),
\end{align*}
where, we have introduced the function ${\map{h_2}{\R_{\geq 0}}{\R_{\geq 0}}}$ defined by $h_2(\cs) := 2\gamma\frac{\amax}{\amin} (\bmax + \frac{\amax}{\gamma} + 2\cs)$.
Substituting the previous bounds on the LMI~\eqref{eq:boundQ12Q22Q12^T} we get
\begin{align*}
Q_{12}Q_{22}^{-1}\trasp{Q_{12}} \preceq (\as)^2\kappa\Big(h_1(\cs) 
    + h_2(\cs) +\gamma^2\frac{\amax}{\amin}\Big)I_n.
\end{align*}
Next, we compute
\begin{align*}
    Q_{11} &= 2B + \Big(\frac{2}{\gamma} - 2\as  \Big)\trasp{A}GA   -  2\cs I_n \\
    &\succeq \Big(2\bmin - \relu\Big(2\as - \frac{2}{\gamma}\Big)\amax  - 2\cs \Big)I_n.
\end{align*}
Finally, we have
\begin{align*}
& Q_{11} {-} Q_{12}Q_{22}^{-1}\trasp{Q_{12}}
\succeq I_n\Big(2\bmin - \relu\Big(2\as - \frac{2}{\gamma}\Big)\amax  - 2\cs\\
& \qquad-(\as)^2\kappa\Big(h_1(\cs) + h_2(\cs) +\gamma^2\frac{\amax}{\amin}\Big)\Big) \succeq 0,
\end{align*}
where the last inequality follows from constraint~\eqref{eq:nonlinear-constraint}.
This concludes the proof.
\end{proof}
\end{arxiv}

\begin{arxiv}
\begin{lemma}\label{lemma:existence-opt}
    Consider the nonlinear program~\eqref{eq:nonlinear-prog}. Then optimal parameters $c^\star, \alpha^\star, \varkappa^\star$ exist, are finite, and are strictly positive. 
\end{lemma}
\begin{proof}
    We have the obvious bounds on $c, \alpha, \varkappa$:
    \begin{align*}
    0 \leq c &\leq c_{\max} :=\min\Big\{\mf, \frac{3}{4}\amin\alpha_{\max}\Big\}, \\
    0 \leq \alpha &\leq \alpha_{\max} := \min\Big\{\frac{1}{\sqrt{\amax}} - \varepsilon, \frac{\gamma}{\amax}\Big\}, \\
    \frac{2}{3} \leq \varkappa&.
\end{align*}
    We can see that at fixed $c,\varkappa$, the function $h$ is continuous and decreasing in $\alpha$ and that for fixed $\alpha,\varkappa$, the function $h$ is continuous and decreasing in $c$. Let $\varkappa > 2/3$ be arbitrary. Pick $\tilde{c} \in {]0, \mf[}$ and note that $h(\tilde{c},0,\varkappa) = 2\mf - 2\tilde{c} > 0$. By continuity of $h$ in $\alpha$, there exists sufficiently small $\alpha > 0$ such that $h(\tilde{c},\alpha,\varkappa) \geq 0$. Then pick $0 < c < \min\{\tilde{c}, (\frac{3}{4} - \frac{1}{2\varkappa})\alpha\amin\}$. By definition, $c  \leq (\frac{3}{4} - \frac{1}{2\varkappa})\alpha\amin$ and $h(c,\alpha,\varkappa) \geq h(\tilde{c},\alpha,\varkappa) \geq 0$ since $h$ is decreasing in $c$ for fixed $\alpha,\varkappa$. In other words, for every $\varkappa > 2/3$, there exist parameter values $(c,\alpha,\varkappa)$ that are all positive and feasible (at $\varkappa = 2/3$, the optimal value of $c$ is $0$). Since $c$ and $\alpha$ are bounded above, we know that optimal values must be finite. We now show that the optimal value of $\varkappa$ is also bounded above. 

    For $\varkappa \geq 2/3$, define the set 
    \begin{align*}
        S_\varkappa := &\Big\{(c,\alpha) \in \realnonnegative \times \realnonnegative \; | \; \alpha \leq \alpha_{\max}, c \leq \Big(\frac{3}{4}-\frac{1}{2\varkappa}\Big)\alpha\amin, \\
        &\quad h(c,\alpha,\varkappa) \geq 0
    \Big\}.
    \end{align*}
    Because $h$ is continuous, we can see that for every $\varkappa \geq 2/3$, $S_\varkappa$ is a closed set. Moreover, we can also see that
    $$S_{\varkappa} \subseteq [0, c_{\max}] \times [0, \alpha_{\max}],$$
    so $S_{\varkappa}$ is a closed subset of a compact set, so it is also compact. Then for every $\varkappa \geq 2/3$ and $(c,\alpha) \in S_{\varkappa}$, the mapping $(c,\alpha) \mapsto c$ is continuous on the compact set $S_{\varkappa}$ and thus attains its maximal value at parameter values $c_{\varkappa}, \alpha_{\varkappa}$, both of which are positive for every $\varkappa > 2/3$ by the above reasoning. Consider specifically $c_1, \alpha_1 > 0$ which are the parameter values which maximize the map over the compact set $S_1$. It is now straightforward to argue that there exists $M > 0$ sufficiently large such that for all $d \geq M$, $c_d < c_1$ and thus implies that the optimal value for $\varkappa$ lies in the bounded interval $[2/3, M]$. To see this intuitively, if $M$ is sufficiently large, $h(c_M,\alpha_M,M) \geq 0$ can be made to imply that $\alpha_d \leq \frac{4}{3}\frac{c_1}{\amin}$ for every $d \geq M$ (by picking $M$ large enough and since $h$ is decreasing in $\varkappa$) and thus $c_d \leq (\frac{3}{4} - \frac{1}{2d})\alpha_d \amin < \frac{3}{4}\alpha_d \amin \leq c_1$. 
    
    Then define the set
    \begin{align*}
        S := &\Big\{(c,\alpha,\varkappa) \in \realnonnegative \times \realnonnegative \times \realnonnegative\; | \; \alpha \leq \alpha_{\max}, \\&\quad c \leq \Big(\frac{3}{4}-\frac{1}{2\varkappa}\Big)\alpha\amin,  h(c,\alpha,\varkappa) \geq 0, 2/3 \leq \varkappa \leq M
    \Big\}.
    \end{align*}
    We can easily see that $S$ is closed and is also compact since
    $$S \subseteq [0,c_{\max}] \times [0,\alpha_{\max}] \times [2/3, M].$$
    Moreover, for every $(c,\alpha,\varkappa) \in S$ the mapping $(c,\alpha,\varkappa) \mapsto c$ is continuous on the compact set $S$ and attains its maximal value at parameter values $c^\star,\alpha^\star,\varkappa^\star$ all of which are positive and finite. This concludes the proof.
\end{proof}
\end{arxiv}

\begin{figure}
    \centering
    \includegraphics[width=0.95\linewidth]{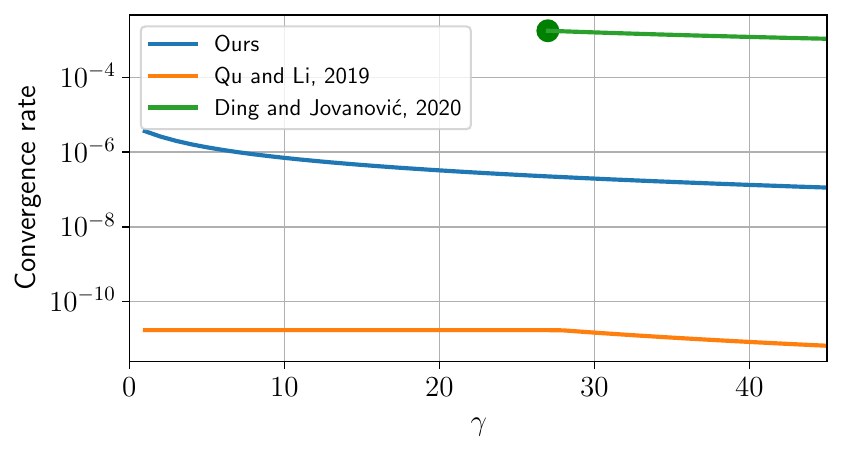}
    \caption{\change{Estimates of contraction or exponential convergence rate of the dynamics~\eqref{eq:prox-aug-primal-dual} with $(\rho,\ell,\amin,\amax) = (1.03, 27.81, 0.1, 1)$ as a function of $\gamma$. We compare the contraction estimates from~\eqref{eq:nonlinear-prog} to the exponential convergence rates from~\cite[Prop.~3]{DD-MRJ:20} and~\cite[Theorem~2]{GQ-NL:19}.}}\label{fig:contraction-rates}
\end{figure}

\change{Finally, we empirically compare our estimate of the contraction rate for the dynamics~\eqref{eq:prox-aug-primal-dual} as a function of $\gamma$ to existing convergence rates in the literature. We compare against~\cite[Prop.~3]{DD-MRJ:20} and~\cite[Theorem~2]{GQ-NL:19}. We elect to compare against~\cite{DD-MRJ:20} rather than~\cite{NKD-SZK-MRJ:19} since the estimate in~\cite{DD-MRJ:20} depends only on}
\change{$\amin,\amax$ while the estimate from~\cite[Theorem~3]{NKD-SZK-MRJ:19} depends on $A$. We note that both~\cite{DD-MRJ:20} and~\cite{GQ-NL:19} establish exponential convergence rather than contractivity, but the proof methods could be extended to establish contraction. In Figure~\ref{fig:contraction-rates}, we plot these convergence estimates.}

\change{We note that the estimate from~\cite[Prop.~3]{DD-MRJ:20} is only valid for $\gamma \geq \ell - \rho$ and this is denoted by the circle in the plot. For $\gamma < \ell - \rho$, our contraction estimate is orders of magnitude better than the estimate in~\cite{GQ-NL:19}. For $\gamma \geq \ell - \rho$, the estimate in~\cite{DD-MRJ:20} is better. In other words, our estimate appears to be the sharpest-known contraction rate which is valid for all $\gamma > 0$. Our estimates may be well-suited for poorly conditioned problems where $\rho \ll \ell$ which would require very large $\gamma$ for the analysis in~\cite{DD-MRJ:20} to apply.}

\bibliography{alias,Main,FB,New}

\end{document}